\documentclass{article}
\usepackage{graphicx} % Required for inserting images
\usepackage{amsmath,amssymb,amsthm,amsfonts} % 数学公式、符号、定理环境
\usepackage{mathrsfs} % 手写体
\usepackage{xparse} % 提供高级命令定义功能
\usepackage{refcount} % 提供获取引用计数器值的功能
\usepackage{algorithm}% 算法环境
\usepackage{algorithmicx}% 算法步骤显示行号
\usepackage{algpseudocode}% 伪代码环境

\makeatletter% 环境定义代码，algorithm环境为浮动体，不能跨页分割，需要重新定义环境
\newenvironment{breakablealgorithm}
{% \begin{breakablealgorithm}
		\begin{center}
			\refstepcounter{algorithm}% New algorithm
			\hrule height.8pt depth0pt \kern2pt% \@fs@pre for \@fs@ruled
			\renewcommand{\caption}[2][\relax]{% Make a new \caption
				{\raggedright\textbf{\ALG@name~\thealgorithm} ##2\par}%
				\ifx\relax##1\relax % #1 is \relax
				\addcontentsline{loa}{algorithm}{\protect\numberline{\thealgorithm}##2}%
				\else % #1 is not \relax
				\addcontentsline{loa}{algorithm}{\protect\numberline{\thealgorithm}##1}%
				\fi
				\kern2pt\hrule\kern2pt
			}
		}{% \end{breakablealgorithm}
		\kern2pt\hrule\relax% \@fs@post for \@fs@ruled
	\end{center}
}
\makeatother
\usepackage[urlcolor=blue]{hyperref}% 用于实现交叉引用的超链接功能，可选，若不需要跳转功能可省略
\usepackage{threeparttable} % 表格注释
\usepackage{setspace} % 行距设置
\usepackage{titlesec} % 标题格式
\usepackage{float} % 浮动体控制
\usepackage{fancyhdr} % 页眉页脚
\usepackage{cite} % 引用格式
\usepackage{tikz} % 绘图
\usepackage[dvipsnames, svgnames]{xcolor}  % 导言区加载
\usetikzlibrary{arrows.meta}
\usetikzlibrary{positioning} % 加载tikz的positioning库，支持相对定位(如above=of ...)
\usetikzlibrary{decorations.pathmorphing} % 加载路径变形库（核心）
\usepackage{pgffor} % 循环功能
\usepackage{indentfirst} % 段落首行缩进
\usepackage{caption} % 自定义图表标题格式（如字体、编号等）

\newtheoremstyle{mytheoremstyle}% 样式名称
  {3pt}% 定理环境与上文的间距
  {3pt}% 定理环境与下文的间距
  {\normalfont}% 正文的字体样式（这里是正常字体）
  {0pt}% 首行缩进
  {\bfseries}% 定理标题的字体样式（这里是加粗）
  {}% 标题后的标点
  {.5em}% 标题与正文之间的间距
  {}% 定理标题的格式（这里为空）
  
\theoremstyle{mytheoremstyle}
\newtheorem{theorem}{Theorem}%设置按章节编号
\newtheorem{lemma}{Lemma}
\newtheorem{proposition}{Proposition}
\newtheorem{example}{Example}

\tikzset{
    base arc/.style = {
        draw = black,            % 对应 arrow1 的 draw=black
        thick,                   % 对应 arrow1 的 thick
        -{Latex[length = 1.5mm, width = 1.5mm]}, % 你之前 arrow1 的 Latex 箭头
        line width=0.8pt        % 原始线宽（可保留）
    },
    bend arc/.style = {
        base arc,               % 继承包含你箭头的基础样式
        bend right=30,           % 原始弯曲角度
        looseness=1.2            % 弯曲弧度松紧度
    },
    vertex/.style = {
        circle,draw,fill=white,inner sep=1.5pt,line width=0.8pt % 黑色边的空心小圆点
    },
    solidvertex/.style = {
        circle,fill = black,minimum size = 5pt,inner sep = 0pt
    },
    zigzag/.style={
        draw = black,            % 和base arc一致：线条颜色
        thick,                   % 和base arc一致：基础线粗
        -{Latex[length = 2.2mm, width = 2.2mm]}, % Latex箭头
        line width=0.8pt,        % 和base arc一致：精准线宽
        decorate,               % 启用装饰
        decoration={
            zigzag,             % 锯齿装饰类型
            segment length=0.1cm, % 你设定的锯齿长度
            amplitude=0.03cm,   % 你设定的锯齿高度
            pre length=0pt,     % 起始无多余长度
            post length=0pt     % 结束无多余长度
        }
    }
}

\newcommand{\labelNode}[3]{
    \node[#1] at (#2) {$#3$};
}

\NewDocumentCommand{\specialproof}{m}{%
  \begingroup
  \def\currentlabel{#1}% 存储当前定理标签
  \renewcommand{\proofname}{Proof of Theorem \ref*{#1}}% 设置证明标题
  \begin{proof}% 开始原始证明环境
}

\NewDocumentCommand{\endspecialproof}{}{%
  \end{proof}% 结束原始证明环境
  \endgroup
}

\NewDocumentCommand{\specialproofpro}{m}{%
  \begingroup
  \def\currentlabel{#1}% 存储当前定理标签
  \renewcommand{\proofname}{Proof of Proposition \ref*{#1}}% 设置证明标题
  \begin{proof}% 开始原始证明环境
}

\NewDocumentCommand{\endspecialproofpro}{}{%
  \end{proof}% 结束原始证明环境
  \endgroup
}

\newcounter{proofcase}
\newcounter{subproofcase}[proofcase]

\newenvironment{ProofCase}[1][]{%
  \refstepcounter{proofcase}%
  \par\smallskip\noindent\textbf{Case \theproofcase. #1}\hspace{0.5em}%
  \ignorespaces
}{%
  \par\smallskip%
}

\newenvironment{SubProofCase}[1][]{%
  \refstepcounter{subproofcase}%
  \par\smallskip\noindent\textbf{Subcase \theproofcase.\arabic{subproofcase}. #1}\hspace{0.5em}%
  \ignorespaces
}{%
  \par\smallskip%
}

\AtBeginEnvironment{proof}{%
  \setcounter{proofcase}{0}%
}

\title{Internally-disjoint directed pendant Steiner trees with three terminal vertices in Cartesian product digraphs}
\author{Shanshan Yu$^{1}$,  Yuefang Sun$^{2}$%\footnote{Corresponding author.}
\\
$^{1}$ School of Mathematics and Statistics,
Ningbo University,\\
Ningbo 315211,  China, yushanshan33hh@163.com\\
$^{2}$ Corresponding author. School of Mathematics and Statistics,\\
Ningbo University,
Ningbo 315211, China, sunyuefang@nbu.edu.cn}
\date{}

\begin{document}

\maketitle

\begin{abstract}
    Let $D=(V(D),A(D))$ be a digraph with a terminal vertex subset $S\subseteq V(D)$ such that $|S|=k\geq 2$. An out-tree $T$ of $D$ rooted at $r$ is called a directed pendant $(S,r)$-Steiner tree (or, pendant $(S,r)$-tree for short) if $r\in S\subseteq V(T)$ and $d_{T}^{+}(r)=d_{T}^{-}(u)=1$ for each $u\in S\backslash \{r\}$. Two pendant $(S,r)$-trees $T_{1}$ and $T_{2}$ are internally-disjoint if $A(T_{1})\cap A(T_{2})=\varnothing$ and $V(T_{1})\cap V(T_{2})=S$. The pendant-tree $k$-connectivity $\tau_{k}(D)$ of $D$ is defined as $$\tau_{k}(D)=\min\{\tau_{S,r}(D)\mid S\subseteq V(D),|S|=k,r\in S\},$$ where $\tau_{S,r}(D)$ denotes the maximum number of pairwise internally-disjoint pendant $(S,r)$-trees in $D$. 
    
    In this paper, we derive a sharp lower bound for the pendant-tree 3-connectivity of the Cartesian product digraph $D\square H$, where $D$ and $H$ are both strong digraphs. Specifically, we prove the lower bound $\tau_{3}(D\square H)\geq \tau_{3}(D)+\tau_{3}(H)$. Moreover, we propose a polynomial-time algorithm for finding internally-disjoint pendant $(S,r)$-trees which attain this lower bound.
    \vspace{0.3cm}
\vspace{0.3cm}
\end{abstract}

\section{Introduction}
Throughout this paper, all digraphs are assumed to have no parallel arcs or loops. We write $[n]$ for the collection of all natural numbers ranging from $1$ to $n$, and let $|S|$ stand for the cardinality of a set $S$. %A directed path containing $n$ vertices is denoted as $\overrightarrow{P_{n}}$. 
A directed path $\overrightarrow{P}$ is referred to as a directed $u-v$ path, denoted by $\overrightarrow{P_{uv}}$, if $u$ is the initial vertex and $v$ is the terminal vertex of $\overrightarrow{P}$. For any $u',v'\in V(\overrightarrow{P_{uv}})$ such that $u'$ precedes $v'$ on $\overrightarrow{P_{uv}}$, we denote by $\overrightarrow{P_{uv}[u',v']}$ the unique directed subpath from $u'$ to $v'$ in $\overrightarrow{P_{uv}}$. 
A digraph $D'$ is called a {\em subdigraph} of $D$ when $V(D')\subseteq V(D)$, $A(D')\subseteq A(D)$, and every arc in $A(D')$ has both end vertices contained in $V(D')$. For any set $X\subseteq V(D)$, the digraph $D-X$ is the subdigraph induced by $V(D)-X$.
A digraph $D=(V(D),A(D))$ is said to be {\em strongly connected} (or {\em strong} for short) if for any pair of distinct vertices $u,v\in V(D)$, there exist both a directed $u-v$ path $\overrightarrow{P_{uv}}$ and a directed $v-u$ path $\overrightarrow{P_{vu}}$. For a strong digraph $D=(V,A)$, a set $S\subset V$ is a {\em separator} if $D-S$ is not strong. A digraph $D$ is {\em $\ell$-strongly connected} (or {\em $\ell$-strong}) if $|V|\geq \ell+1$ and $D$ has no separator with fewer than $\ell$ vertices.
A digraph $D$ is {\em symmetric} if the presence of an arc $uv \in A(D)$ implies that $vu \in A(D)$ as well. Equivalently, a symmetric digraph $D$ can be constructed from its {\em underlying undirected graph} $G$ by replacing each edge of $G$ with a pair of opposite arcs, and we denote this by $D=\overleftrightarrow{G}$.

%The {\em connectivity} of an undirected graph $G$, written $\kappa(G)$, is the minimum order of a vertex set $S\subseteq V(G)$ such that $G\backslash S$ is disconnected or has only one vertex. The {\em generalized connectivity} (or $k$-tree connectivity) of a graph $G$, introduced by Hager~\cite{Hager}, is a natural generalization of classical connectivity. Another tree-connectivity parameter called pendant tree connectivity was also proposed in his work \cite{Hager}. 
For an undirected graph $G=(V(G),E(G))$ and a terminal vertex subset $S\subseteq V(G)$ satisfying $|S|=k\geq 2$, an {\em $S$-Steiner tree} (or {\em $S$-tree} for short) $T$ refers to a tree in $G$, which includes all vertices of $S$. Hager~\cite{Hager} first introduced the {\em pendant $S$-Steiner tree} (or {\em pendant $S$-tree} for short) in 1985. This is a special $S$-tree where every vertex in $S$ is a leaf (i.e., has degree one) in $T$. Two pendant $S$-trees are said to be {\em edge-disjoint} if they share no common edges. Additionally, if the only vertices shared by two edge-disjoint pendant $S$-trees are exactly those in $S$, then these trees are called {\em internally-disjoint}. The {\em local pendant-tree $k$-connectivity of $S$}, denoted by $\tau_{S}(G)$, is the maximum number of pairwise internally-disjoint pendant $S$-trees in $G$. We further define the {\em pendant-tree $k$-connectivity} of $G$ as
\begin{align*}
    \tau_{k}(G)=\min\{\tau_{S}(G)\mid S\subseteq V(G),|S|=k\}.
\end{align*}
From the definition above, it is straightforward to see that $ \tau_{2}(G)=\kappa(G)$, where $\kappa(G)$ is the classical vertex connectivity of $G$. For more information on pendant tree connectivity and related topics, the readers can see~\cite{Li-Mao} for a monograph.

An {\em out-tree} is an oriented tree where all vertices except one, called the {\em root}, have in-degree one. For a digraph $D=(V(D),A(D))$, and a terminal vertex subset $S\subseteq V(D)$ with $r\in S$ and $|S|=k\geq 2$, a {\em directed $(S,r)$-Steiner tree} (or {\em $(S,r)$-tree} for short), is an out-tree $T$ rooted at $r$ and contains all vertices of $S$. If every vertex in $S$ has degree one in an $(S,r)$-tree $T$, then $T$ is called a {\em pendant $(S,r)$-tree}. Specifically, in each pendant $(S,r)$-tree, the in-degree of $r$ and the out-degree of each vertex $v\in S\setminus \{r\}$ equal zero, and the out-degree of $r$ and the in-degree of each vertex $v\in S\setminus \{r\}$ is one~\cite{Yu-Sun-2}. Two pendant $(S,r)$-trees $T_{1}$ and $T_{2}$ are said to be {\em internally-disjoint} if they share no arcs and their only common vertices are those in $S$. Let $\tau_{S,r}(D)$ be the maximum number of pairwise internally-disjoint pendant $(S,r)$-trees in $D$, the {\em directed pendant-tree $k$-connectivity} of a digraph $D$ is defined as
\begin{align*}
    \tau_{k}(D)=\min\{\tau_{S,r}(D)\mid S\subseteq V(D),|S|=k,r\in S\}.
\end{align*}
By definition, it is straightforward to see that $\tau_{2}(D)=\kappa(D)$, where $\kappa(D)$ is the classical {\em vertex connectivity} of a digraph $D$. This equality confirms that the directed pendant-tree $k$-connectivity is a natural extension of classical vertex connectivity of digraphs. Moreover, this parameter is closely linked to the other parameters of digraphs, including directed tree connectivity~\cite{Sun-Yeo}, strong subgraph connectivity~\cite{Sun-Gutin, Sun-GutinJOIN, Sun-Gutin-Yeo-Zhang}, and directed cycle connectivity~\cite{Wang-Sun}. For more information on these topics, the readers can see~\cite{Sun-book} for a new monograph.

The pendant-tree $k$-connectivity is not only a natural combinatorial measure, but also motivated by its meaningful practical interpretations. For example, let $D$ represent a network. If one aims to connect two vertices in $D$, a path suffices for this connection. However, if one intends to connect a vertex set $S\subseteq V(D)$ with $|S|\geq 3$, a tree is required to link all vertices in $S$, and such a tree is commonly referred to as a Steiner tree. A pendant Steiner tree (an important variant of the Steiner tree) further restricts terminal vertices to be leaves of the tree, and this structure is widely adopted in the physical design of VLSI circuits \cite{Gr¨otschel-Martin-Weismantel, Sherwani} and the fault-tolerant routing of multi-terminal communication networks \cite{Chen-Hsieh}.
%Such a tree connecting a vertex set is generally called a (pendant) Steiner tree, which is widely applied in the physical design of VLSI circuits.

%Two directed $uv$ paths $\overrightarrow{P_{1}}$ and $\overrightarrow{P_{2}}$ are called {\em internally-disjoint} if their common vertex set is exactly $\{u,v\}$. Let $B\subseteq V(D)$ and $A\subset V(D)\backslash B$. The set of directed $(A,B)$ paths is a family of internally-disjoint directed paths whose starting at a vertex $u\in A$, ending at a vertex $v\in B$ and whose internal vertices belong neither to $A$ nor to $B$. 

Product digraphs, which construct new composite digraphs by combining two or more base digraphs via specific rules, are a powerful tool for designing and analyzing large-scale interconnection networks. Among the various types of product digraphs, the {\em Cartesian product} digraph forms a fundamental class, widely adopted for their construction owing to its excellent scalability and symmetry. In recent years, this class of digraphs has attracted growing research interest, as seen in \cite{Bang-Jensen-Gutin-1, Bao-Igarashi-Öhring}.

The Cartesian product of two digraphs $D$ and $H$, denoted by $D\square H$, is a digraph with vertex set
\begin{align*}
    V(D\square H)=V(D)\times V(H)=\{(u,v)\mid u\in V(D), v\in V(H)\},
\end{align*}
and arc set
\begin{align*}
    A(D\square H)=\{(u,v)(u',v')\mid uu'\in A(D), v=v'~\text{or}~u=u',vv'\in A(H)\}.
\end{align*}
Clearly, this product is commutative, that is $D\square H\cong H\square D$. Figure \ref{fig1} illustrates this construction with the example $\overrightarrow{P_{4}}\square \overrightarrow{C_{3}}$.
\begin{figure}[H] 
    \centering
    \begin{tikzpicture}
        \begin{scope}[xshift = 0cm]
            \node[vertex] (u1) at (0,0) {};
            \node[vertex] (u2) at (0,0.8) {};
            \node[vertex] (u3) at (0,1.6) {};
            \node[vertex] (u4) at (0,2.4) {};
            \node[vertex] (v1) at (1,0) {};
            \node[vertex] (v2) at (1.8,0) {};
            \node[vertex] (v3) at (2.6,0) {};
            \draw [base arc] (u1) -- (u2);
            \draw [base arc] (u2) -- (u3);
            \draw [base arc] (u3) -- (u4);
            \draw [base arc] (v1) -- (v2);
            \draw [base arc] (v2) -- (v3);
            \draw [bend arc] (v3) to (v1);
            \node at (1.2,-0.5) {(a) $\overrightarrow{P_{4}}$ and $\overrightarrow{C_{3}}$};
        \end{scope}
        \begin{scope}[xshift = 5cm]
            \node[vertex] (x11) at (0,0) {};
            \node[vertex] (x12) at (0.8,0) {};
            \node[vertex] (x13) at (1.6,0) {};
            \node[vertex] (x21) at (0,0.8) {};
            \node[vertex] (x22) at (0.8,0.8) {};
            \node[vertex] (x23) at (1.6,0.8) {};
            \node[vertex] (x31) at (0,1.6) {};
            \node[vertex] (x32) at (0.8,1.6) {};
            \node[vertex] (x33) at (1.6,1.6) {};
            \node[vertex] (x41) at (0,2.4) {};
            \node[vertex] (x42) at (0.8,2.4) {};
            \node[vertex] (x43) at (1.6,2.4) {};
            \draw [base arc] (x11) -- (x12);
            \draw [base arc] (x12) -- (x13);
            \draw [bend arc] (x13) to [bend right=30] (x11);
            \draw [base arc] (x21) -- (x22);
            \draw [base arc] (x22) -- (x23);
            \draw [bend arc] (x23) to [bend right=30] (x21);
            \draw [base arc] (x31) -- (x32);
            \draw [base arc] (x32) -- (x33);
            \draw [bend arc] (x33) to [bend right=30] (x31);
            \draw [base arc] (x41) -- (x42);
            \draw [base arc] (x42) -- (x43);
            \draw [bend arc] (x43) to [bend right=30] (x41);
            \draw [base arc] (x11) -- (x21);
            \draw [base arc] (x21) -- (x31);
            \draw [base arc] (x31) -- (x41);
            \draw [base arc] (x12) -- (x22);
            \draw [base arc] (x22) -- (x32);
            \draw [base arc] (x32) -- (x42);
            \draw [base arc] (x13) -- (x23);
            \draw [base arc] (x23) -- (x33);
            \draw [base arc] (x33) -- (x43);
            \node at (0.8,-0.5) {(b) $\overrightarrow{P_{4}}\square \overrightarrow{C_{3}}$};
        \end{scope}
    \end{tikzpicture}
    \caption{The Cartesian product $\overrightarrow{P}_{4}\square \overrightarrow{C}_{3}$}
    \label{fig1}
\end{figure}
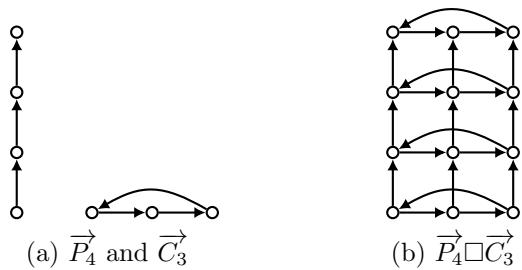

In this paper, we study the directed pendant-tree
$k$-connectivity of Cartesian product digraphs, and give a sharp lower bound for $\tau_{3}(D\square H)$ in terms of $\tau_{3}(D)$ and $\tau_{3}(H)$ as follows. 
\begin{theorem}\label{thm1.1} % 还缺紧的例子
    Let $D$ and $H$ be two strong digraphs, we have
    \begin{align*}
        \tau_{3}(D\square H)\geq \tau_{3}(D)+\tau_{3}(H).
        %\tau_{3}(D\square H)\geq \min \{3\lfloor\frac{\tau_{3}(D)}{2}\rfloor,3\lfloor\frac{\tau_{3}(H)}{2}\rfloor\}.
    \end{align*}
    Moreover, this bound is sharp.
\end{theorem}

Furthermore, we propose a polynomial-time algorithm for finding internally-disjoint pendant $(S,r)$-trees which attain this lower bound.
    
\begin{theorem}\label{thm1.2}
    Let $D\square H$ be a Cartesian product digraph, and let $S=\{r,x,y\}\subseteq V(D\square H)$ be a terminal set with $r$ as the root. Let $S_{D}=\{u_{p},u_{q},u_{w}\}$
    and $u_{p}\in S_{D}$ is the corresponding vertex of $r$ in $D$. Similarly, let $S_{H}=\{v_{a},v_{b},v_{c}\}$
    and $v_{a}\in S_{H}$ is the corresponding vertex of $r$ in $H$. Suppose that $\widetilde{T}_{1},\widetilde{T}_{2},\cdots,\widetilde{T}_{\ell}$ are $\ell$ internally-disjoint pendant $(S_{D},u_{p})$-trees in $D$, and $\hat{T}_{1},\hat{T}_{2},\cdots,\hat{T}_{h}$ are $h$ internally-disjoint pendant $(S_{H},v_{a})$-trees in $H$. Algorithm \ref{Alg-1} constructs a family of $\ell+h$ pairwise internally-disjoint pendant $(S,r)$-trees in $D\square H$ in time $O(|V(D)|^{2}|A(D)|+|V(H)|^{2}|A(H)|)$.
\end{theorem}

%In particular, if a strong digraph is a symmetric digraph, then we have the following result.
%\begin{theorem}\label{thm1.2} %还没有很确定
    %Let $\overleftrightarrow{G}$ be a onnected symmetric digraph and $H$ be a strong digraph, we have
    %\begin{align*}
        %\tau_{3}(\overleftrightarrow{G}\square H)\geq \tau_{3}(\overleftrightarrow{G})+\tau_{3}(H).
    %\end{align*}
    %Moreover, the bound is sharp.
%\end{theorem}
The rest of this paper is organized as follows. Some necessary definitions and notations are given in Section 2. The proof of Theorem~\ref{thm1.1} and an example illustrating the tightness of the bound are shown in Section 3. In section 4, we propose a polynomial-time algorithm to find a family of internally-disjoint pendant $(S,r)$-trees which attain the lower bound described in Theorem~\ref{thm1.1}. As shown in Theorem~\ref{thm1.2}, the time complexity of this algorithm is $O((|V(D)|^{2}|A(D)|+|V(H)|^{2}|A(H)|)$ which is clearly polynomial in $|V(D)|$ and $|V(H)|$.

\section{Preliminaries}
In this section, some basic definitions and results that will be used in this paper are introduced.

Let $D=(V(D),A(D))$ be a digraph. For any vertex $v\in V(D)$, we use $N_{D}^{+}(v)$ (resp. $N_{D}^{-}(v)$) to represent the {\em out-neighbours} (resp. {\em in-neighbours}) of $v$ in $D$. More precisely, 
\begin{align*}
    N_{D}^{+}(v)=\{u\in V(D)\backslash \{v\}\mid vu\in A(D)\}, N_{D}^{-}(v)=\{w\in V(D)\backslash \{v\}\mid wv\in A(D)\}.
\end{align*}
The {\em out-degree} (resp. {\em in-degree}) of $v$ in $D$ is defined as $d_{D}^{+}(v)=|N_{D}^{+}(v)|$ (resp. $d_{D}^{-}(v)=|N_{D}^{-}(v)|$). We refer the out-degree and in-degree of a vertex as its {\em semi-degrees}. The {\em degree} of a vertex $v$ in $D$ is the sum of its semi-degrees, that is, $d_{D}(v)=d_{D}^{+}(v)+d_{D}^{-}(v)$. If the context is clear, we always omit $D$ in the above notation. The {\em minimum out-degree} (resp. {\em minimum in-degree}) of $D$ is $\delta^{+}(D)=\min\{d^{+}(v)\mid v\in V(D)\}$ (resp. $\delta^{-}(D)=\min\{d^{-}(v)\mid v\in V(D)\}$). The {\em minimum semi-degree} of $D$ is $\delta^{0}(D)=\min\{{\delta}^{+}(D),{\delta}^{-}(D)\}$.

Let $D$ and $H$ be two strong digraphs with vertex sets $V(D)=\{u_{i}\mid i\in[n]\}$ and $V(H)=\{v_{j}\mid j\in [m]\}$, respectively. Unless specified otherwise, we use this labeling for vertices in $D$ and $H$ throughout the paper. %Then $V(D\square H)=\{(u_{i},v_{j})\mid i\in [n], j\in[m])\}$. A $D$-layer $D(v)=\langle\{(u_{i},v)\mid i\in [n]\}\rangle$ is the induced subgraph of $D\square H$ related to a fixed vertex $v$ of $H$, which is isomorphic to $D$. Analogously, an $H$-layer $H(u)=\langle\{(u,v_{j})\mid j\in [m]\}\rangle$ is the induced subgraph of $D\square H$ related to a fixed vertex $u$ of $D$, which is isomorphic to $H$.
We first define two types of induced subdigraphs of $D\square H$.
\begin{itemize}
    \item For each $j\in [m]$, let $D(v_{j})$ be the subdigraph induced by the vertex set $\{(u_{i},v_{j})\mid i\in [n]\}$;
    \item For each $i\in [n]$, let $H(u_{i})$ be the subdigraph induced by the vertex set $\{(u_{i},v_{j})\mid j\in [m]\}$.
\end{itemize}
Clearly, for each $j\in [m]$, $D(v_{j})\cong D$; for each $i\in [n]$, $H(u_{i})\cong H$.
%Note that $D(u_{i},v_{j_{1}})\cong D(u_{i},v_{j_{2}})$ for any distinct $v_{j_{1}},v_{j_{2}}\in V(H)$. For simplicity, we abbreviate $D(u_{i},v_{j})$ to $D(v_{j})$. Similarly, $H(u_{i},v_{j})$ can be abbreviated to $H(u_{i})$.

For any $u_{i_{1}},u_{i_{2}}\in V(D)$ and $v_{j_{1}},v_{j_{2}}\in V(H)$, the following vertex inclusions hold
\begin{itemize}
    \item $(u_{i_{1}},v_{j_{1}}),(u_{i_{1}},v_{j_{2}})\in V(H(u_{i_{1}}))$; $(u_{i_{2}},v_{j_{1}}),(u_{i_{2}},v_{j_{2}})\in V(H(u_{i_{2}}))$;
    \item $(u_{i_{1}},v_{j_{1}}),(u_{i_{2}},v_{j_{1}})\in V(D(v_{j_{1}}))$; $(u_{i_{1}},v_{j_{2}}),(u_{i_{2}},v_{j_{2}})\in V(D(v_{j_{2}}))$.
\end{itemize}
We refer to $(u_{i_{1}},v_{j_{2}})$ as {\em the corresponding vertex} of $(u_{i_{1}},v_{j_{1}})$ in $H(u_{i_{1}})$, and to $(u_{i_{2}},v_{j_{1}})$ as the corresponding vertex of $(u_{i_{1}},v_{j_{1}})$ in $D(v_{j_{1}})$.
%We use $D(u_{i},v_{j})$ to denote the subdigraph of $D\square H$ induced by the vertex set $\{(u_{i},v_{j})\mid i\in [n]\}$. Similarly, we use $H(u_{i},v_{j})$ to denote the subdigraph of $D\square H$ induced by the vertex set $\{(u_{i},v_{j})\mid j\in [m]\}$. It is easy to see $D(u_{i},v_{j_{1}})=D(u_{i},v_{j_{2}})$ for different $v_{j_{1}}$ and $v_{j_{2}}$ of $H$. Thus, we can replace $D(u_{i},v_{j})$ by $D(v_{j})$ for simplicity. Similarly, we can replace $H(u_{i},v_{j})$ by $H(u_{j})$. 

For any directed path $\overrightarrow{P}\subseteq D$ and directed tree $\widetilde{T}\subseteq D$, their counterparts in $D(v_{j})$ are denoted by $\overrightarrow{P^{(v_{j})}}$ and $\widetilde{T}^{(v_{j})}$, respectively. For any directed path $\overrightarrow{Q}\subseteq H$ and directed tree $\hat{T}\subseteq H$, their counterparts in $H(u_{i})$ are denoted by $\overrightarrow{Q^{(u_{i})}}$ and $\hat{T}^{(u_{i})}$. Furthermore, for distinct vertices $(u_{i},v_{j_{1}}),(u_{i},v_{j_{2}})\in H(u_{i})$, the directed path $\overrightarrow{Q_{(u_{i},v_{j_{1}})(u_{i},v_{j_{2}})}}$ (contained in $H(u_{i})$) can be simplified to $\overrightarrow{Q_{v_{j_{1}}v_{j_{2}}}^{(u_{i})}}$. Analogously, we can replace $\overrightarrow{P_{(u_{i_{1}},v_{j})(u_{i_{2}},v_{j})}}$ (contained in $D(v_{j})$) by $\overrightarrow{P_{u_{i_{1}}u_{i_{2}}}^{(v_{j})}}$. For short, the directed subpath $\overrightarrow{P_{(u_{i_{1}},v_{j})(u_{i_{2}},v_{j})}[(u',v_{j}),(u'',v_{j})]}$ (resp. $\overrightarrow{Q_{(u_{i},v_{j_{1}})(u_{i},v_{j_{2}})}[(u_{i},v'),(u_{i},v'')]}$) is denoted by $\overrightarrow{P_{u_{i_{1}}u_{i_{2}}}^{(v_{j})}[u',u'']}$ (resp. $\overrightarrow{Q_{v_{j_{1}}v_{j_{2}}}^{(u_{i})}[v',v'']}$). %where $u',u''\in V(\overrightarrow{P_{u_{i_{1}}u_{i_{2}}}^{(v_{j})}})$ (resp. $v',v''\in V(\overrightarrow{Q_{v_{j_{1}}v_{j_{2}}}^{(u_{i})}})$) and $u'$
Let $S=\{x_{1},x_{2},x_{3}\}\subseteq V(D)$ and $T$ be a pendant $(S,x_{1})$-tree in $D$. If $\overrightarrow{P_{uv}}\subseteq T$, where $\overrightarrow{P_{uv}}$ is the unique directed path from $u$ to $v$ in $T$, then we use $T[u,v]$ to denote $\overrightarrow{P_{uv}}$. Furthermore, we denote by $\alpha$ the {\em branch vertex} of $T$, i.e. the unique vertex satisfying $V(T[\alpha,x_{2}])\cap V(T[\alpha,x_{3}])=\{\alpha\}$.

Finally, for a vertex set $S=\{(u_{i_{1}},v_{j_{1}}),(u_{i_{2}},v_{j_{2}}),\cdots,(u_{i_{k}},v_{j_{k}})\}\subseteq V(D\square H)$, we define
\begin{itemize}
    \item $S_{D}=\{u_{i_{1}},u_{i_{2}},\cdots,u_{i_{k}}\}$ as the {\em projection} of $S$ onto $D$;
    \item $S_{H}=\{v_{j_{1}},v_{j_{2}},\cdots,v_{j_{k}}\}$ as the {\em projection} of $S$ onto $H$.
\end{itemize}
%we can define the directed path and the directed tree corresponding to

We next state two lemmas relating $\tau_{k}(D)$ to the vertex connectivity $\kappa(D)$ and {\em minimum semi-degree} $\delta^{0}(D)$ of a strong digraph $D$.
\begin{lemma}\cite{Yu-Sun-1}\label{lem2.1}
    Let $k\geq 2$ and $\ell\geq 1$ be two integers. If $D$ is a strong digraph with $\tau_{k}(D)\geq \ell$, then $\kappa(D)\geq k+\ell-2$.
\end{lemma}

\begin{lemma}\cite{Yu-Sun-1}\label{lem2.2}
    Let $k\geq 3$ and $\ell\geq 1$ be two integers. If $D$ is a strong digraph with $\tau_{k}(D)\geq \ell$, then $\delta^{0}(D)\geq k+\ell-1$.
\end{lemma}

In order to show our main result, we first recall Menger's Theorem and the Fan Lemma. These are well-known results about the properties on $\ell$-strong digraphs, and they will be used repeatedly in subsequent arguments.
\begin{theorem}\label{thm2.1}
    (\textbf{Menger's Theorem} \cite{Bang-Gutin}) Let $D$ be an $\ell$-strong digraph, and let $u,v$ be a pair of distinct vertices in $D$. Consequently, there exist $\ell$ pairwise internally-disjoint directed $u-v$ paths $\overrightarrow{P_{uv,1}},\overrightarrow{P_{uv,2}},\cdots,\overrightarrow{P_{uv,\ell}}$ in $D$, and $\ell$ pairwise internally-disjoint directed $v-u$ paths $\overrightarrow{P_{vu,1}'},\overrightarrow{P_{vu,2}'},\cdots,\overrightarrow{P_{vu,\ell}'}$ in $D$.
\end{theorem}
Note that the directed $u-v$ paths and $v-u$ paths above are not required to be internally-disjoint from each other.

\begin{lemma}\label{lem2.3}
     (\textbf{Fan Lemma} \cite{Bang-Jensen-Gutin-1}, pp. 98) Let %$\ell$ be an integer. If 
     $D=(V(D),A(D))$ be an $\ell$-strong digraph with $u\in V(D)$ and let $Z=\{z_{1},z_{2},\cdots,z_{\ell}\}\subseteq V(D)\backslash \{u\}$ be a vertex set of cardinality $\ell$, then there exists an $\ell$-fan, denote by $F_{Z,u}^{-}=\{\overrightarrow{R}_{i}\mid z_{i}-u\}_{i=1}^{\ell}$, from $Z$ to $u$, i.e., there exists a family of $\ell$ pairwise internally-disjoint paths $\overrightarrow{R_{1}},\overrightarrow{R_{2}},\cdots,\overrightarrow{R_{\ell}}$ such that $\overrightarrow{R_{i}}$ is a directed $z_{i}-u$ path.
\end{lemma}

\section{Proof of Theorem \ref{thm1.1}}
In this section, we first establish a proposition, which will be used in the proof of Theorem \ref{thm1.1}. %For the sake of exposition, we denote pairwise internally-disjoint directed paths as PIDDPs, pairwise internally-disjoint directed paths from $u$ to $v$ as $(u,v)$-PIDDPs, and pairwise internally-disjoint pendant $(S,r)$-trees as $(S,r)$-PIDPTs.
Let $D$ and $H$ be two strong digraphs, and let $S=\{r,x,y\}\subseteq V(D\square H)$. We consider the case where the projection of $S$ onto $D$ satisfies $S_{D}=\{u_{p},u_{w}\}\subseteq V(D)$ with $p\neq w$, and the projection of $r$ onto $D$ is exactly $u_{p}$. For any directed path $\overrightarrow{P_{u_{p}u_{w}}}$ in $D$, we show that $\overrightarrow{P_{u_{p}u_{w}}}\square H$ always contains $\tau_{3}(H)$ pairwise internally-disjoint pendant $(S,r)$-trees.
%derive a result concerning the lower bound of $\tau_{3}(\overrightarrow{P_{u_{p}u_{w}}}\square H)$, 
This result is formalized in the following proposition.
\begin{proposition}\label{pro3.1}
    Let $D$ and $H$ be two strong digraphs. For any terminal vertex set $S=\{r,x,y\}\subseteq V(D\square H)$, if $S_{D}=\{u_{p},u_{w}\}$ and the projection of $r$ onto $D$ is exactly $u_{p}$, then there exist $\tau_{3}(H)$ pairwise internally-disjoint pendant $(S,r)$-trees in $\overrightarrow{P_{u_{p}u_{w}}}\square H$.
    %and let $\overrightarrow{P_{u_{p}u_{w}}}$ be a directed $u_{p}-u_{w}$ path in $D$. For the terminal set $S=\{r,x,y\}\subseteq V(\overrightarrow{P_{u_{p}u_{w}}}\square H)$, where the projection of $r$ onto $D$ is exactly $u_{p}$. We have
    %\begin{align*}
        %\tau_{3}(\overrightarrow{P_{u_{p}u_{w}}}\square H)\geq \tau_{3}(H).
    %\end{align*}
    %Moreover, the bound is sharp.
\end{proposition}

\begin{proof}
    Without loss of generality, let $\tau_{3}(H)=h\geq 1$. The argument is divided into the following two cases.
    \begin{ProofCase}
        $\{r,x\}\subseteq V(H(u_{p}))$ and $y\in V(H(u_{w}))$. The case $\{r,y\}\subseteq V(H(u_{p}))$ and $x\in V(H(u_{w}))$ is analogous, we thus omit the details. %Without loss of generality, 
        Suppose $r=(u_{p},v_{1})$, $x=(u_{p},v_{2})$, $y=(u_{w},v_{c})$. %The proof is split into three subcases.
        In the following argument, we can see that this assumption does not affect the correctness of our proof.
        \begin{SubProofCase}
            $(u_{p},v_{c})\notin \{r,x\}$. %where $(u_{p},v_{c})\in V(H(u_{p}))$ is the corresponding vertex of $y=(u_{w},v_{c})\in V(H(u_{w}))$. 
            Let $S'=\{r,x,(u_{p},v_{c})\}$. Since $H({u_{p}})\cong H$, we have $\tau_{3}(H({u_{p}}))=h$, so there exist $h$ pairwise internally-disjoint pendant $(S',r)$-trees in $H(u_{p})$, denoted by $T_{1}^{(u_{p})},T_{2}^{(u_{p})},\cdots,T_{h}^{(u_{p})}$. For each $s\in [h]$, let $(u_{p},v_{j_{s}})$ be the in-neighbor of $(u_{p},v_{c})$ in $T_{s}^{(u_{p})}$, where $v_{j_{s}}\notin \{v_{1},v_{2}\}$ by the definition of a pendant tree. We construct a tree $T_{s}'$ as follows:
            \begin{itemize}
                \item $V(T_{s}')=V(T_{s}^{(u_{p})})\cup V(\overrightarrow{P_{u_{p}u_{w}}^{(v_{j_{s}})}})\cup \{y\}\backslash \{(u_{p},v_{c})\}$;
                \item $A(T_{s}')=A(T_{s}^{(u_{p})})\cup A(\overrightarrow{P_{u_{p}u_{w}}^{(v_{j_{s}})}})\cup \{(u_{w},v_{j_{s}})y\}\backslash \{(u_{p},v_{j_{s}})(u_{p},v_{c})\}$.
            \end{itemize}
        \end{SubProofCase}
        %$(u_{p},v_{c})\in \{r,x\}$. 
        \begin{SubProofCase}
            $(u_{p},v_{c})=x$. By Lemma \ref{lem2.1}, we have $\kappa(H(u_{p}))\geq \tau_{3}(H(u_{p}))+1=h+1$, which implies that $H(u_{p})$ is an $(h+1)$-strong digraph. By Theorem \ref{thm2.1}, there exist $h+1$ pairwise internally-disjoint directed $r-x$ paths in $H(u_{p})$, denoted by $\overrightarrow{Q_{rx,1}},\overrightarrow{Q_{rx,2}},\cdots,\overrightarrow{Q_{rx,h+1}}$. At most one of these paths consists of only a single arc and we denote this directed path as $\overrightarrow{Q_{rx,h+1}}=rx$. All the remaining directed paths contain internal vertices, and let $(u_{p},v_{j_{s}})$ be the in-neighbor of $x$ in $\overrightarrow{Q_{rx,s}}$ for each $s\in [h]$, where $v_{j_{s}}\notin \{v_{1},v_{2}\}$. For each $s\in [h]$, we construct a tree $T_{s}'$ as follows:
            \begin{itemize}
                \item $V(T_{s}')=V(\overrightarrow{Q_{rx,s}})\cup V(\overrightarrow{P_{u_{p}u_{w}}^{(v_{j_{s}})}})\cup \{y\}$;
                \item $A(T_{s}')=A(\overrightarrow{Q_{rx,s}})\cup A(\overrightarrow{P_{u_{p}u_{w}}^{(v_{j_{s}})}})\cup \{(u_{w},v_{j_{s}})y\}$.
            \end{itemize}
            %It can be verified that these trees are $h$ pairwise internally-disjoint pendant $(S,r)$-trees.
        \end{SubProofCase}

        \begin{SubProofCase}
            $(u_{p},v_{c})=r$. %let $(u_{w},v_{f_{i}})$ be the vertex corresponding to $(u_{p},v_{f_{i}})$ in $H(u_{w})$. 
            By the same argument as in Subcase 1.2, we obtain $\overrightarrow{Q_{rx,1}},\overrightarrow{Q_{rx,2}},\cdots,\overrightarrow{Q_{rx,h+1}}$. By Lemmas~\ref{lem2.1} and~\ref{lem2.3}, there exists an $(h+1)$-fan including $h+1$ pairwise internally-disjoint directed paths, denoted by $\overrightarrow{R_{s}^{(u_{w})}}$ for $s\in [h+1]$, from $(u_{w},v_{j_{1}}),(u_{w},v_{j_{2}}),\cdots,(u_{w},v_{j_{h}}), (u_{w},v_{2})$ to $y$ in $H(u_{w})$, where $\overrightarrow{R_{s}^{(u_{w})}}$ is the directed $(u_{w},v_{j_{s}})-y$ path for $s\in[h]$, and $\overrightarrow{R_{h+1}^{(u_{w})}}$ is the directed $(u_{w},v_{2})-y$ path. For each $s\in [h]$, we construct a tree $T_{s}'$ as follows:%$\overrightarrow{R_{v_{j_{s}}v_{1},1}^{(u_{w})}},\overrightarrow{R_{v_{j_{s}}v_{1},2}^{(u_{w})}},\cdots,\overrightarrow{R_{v_{j_{s}}v_{1},\ell+1}^{(u_{w})}}$. 
            \begin{itemize}
                \item $V(T_{s}')=V(\overrightarrow{Q_{rx,s}})\cup V(\overrightarrow{P_{u_{p}u_{w}}^{(v_{j_{s}})}})\cup V(\overrightarrow{R_{s}^{(u_{w})}})$;
                \item $A(T_{s}')=A(\overrightarrow{Q_{rx,s}})\cup A(\overrightarrow{P_{u_{p}u_{w}}^{(v_{j_{s}})}})\cup A(\overrightarrow{R_{s}^{(u_{w})}})$.
            \end{itemize}
        \end{SubProofCase}
        It can be verified that across both subcases, the trees $T_{1}',T_{2}'\cdots,T_{h}'$ are $h$ pairwise internally-disjoint pendant $(S,r)$-trees.
    \end{ProofCase}
    \begin{ProofCase}
        $r\in V(H(u_{p}))$ and $\{x,y\}\subseteq V(H(u_{w}))$. Without loss of generality, let $r=(u_{p},v_{a})$, $x=(u_{w},v_{1})$, $y=(u_{w},v_{2})$. %The proof is split into two subcases.
        \begin{SubProofCase}
            $r\notin \{(u_{p},v_{1}),(u_{p},v_{2})\}$. Let $S''=\{r,(u_{p},v_{1}),(u_{p},v_{2})\}$. Since $\tau_{3}(H(u_{p}))=h$, there exist $h$ pairwise internally-disjoint pendant $(S'',r)$-trees in $H(u_{p})$, denoted by $T_{1}^{(u_{p})},T_{2}^{(u_{p})},\cdots,T_{h}^{(u_{p})}$. For each $s\in [h]$, let $(u_{p},v_{j_{s}})$ be the branch vertex of $T_{s}^{(u_{p})}$. We construct a tree $T_{s}'$ as follows:
            \begin{itemize}
                \item $V(T_{s}')=V(T_{s}^{(u_{p})}[v_{a},v_{j_{s}}])\cup V(\overrightarrow{P_{u_{p}u_{w}}^{(v_{j_{s}})}})\cup V(T_{s}^{(u_{w})}[v_{j_{s}},v_{1}])\cup V(T_{s}^{(u_{w})}[v_{j_{s}},v_{2}])$;
                \item $A(T_{s}')=A(T_{s}^{(u_{p})}[v_{a},v_{j_{s}}])\cup A(\overrightarrow{P_{u_{p}u_{w}}^{(v_{j_{s}})}})\cup A(T_{s}^{(u_{w})}[v_{j_{s}},v_{1}])\cup A(T_{s}^{(u_{w})}[v_{j_{s}},v_{2}])$.
            \end{itemize}
            \begin{SubProofCase}
                $r=(u_{p},v_{1})$. The case of $r=(u_{p},v_{2})$ is similar, and we omit the details. Analogous to Subcase 1.2, we obtain $h+1$ internally-disjoint directed $v_{a}-v_{2}$ paths, denoted by $\overrightarrow{Q_{v_{a}v_{2},1}^{(u_{p})}},\overrightarrow{Q_{v_{a}v_{2},2}^{(u_{p})}},\cdots,\overrightarrow{Q_{v_{a}v_{2},h+1}^{(u_{p})}}$. At most one of these paths consists of only a single arc and we denote this directed path as $\overrightarrow{Q_{v_{a}v_{2},h+1}^{(u_{p})}}=r(u_{p},v_{2})$. All the remaining directed paths contain internal vertices, and let $(u_{p},v_{j_{s}})$ be the in-neighbor of $(u_{p},v_{2})$ in $\overrightarrow{Q_{v_{a}v_{2},h}^{(u_{p})}}$ for each $s\in [h]$, where $v_{j_{s}}\notin \{v_{a},v_{2}\}$. By Lemmas~\ref{lem2.1} and~\ref{lem2.3}, there exists an $(h+1)$-fan including $h+1$ pairwise internally-disjoint directed paths, denoted by $\overrightarrow{R_{s}^{(u_{w})}}$ for $s\in [h+1]$, from $(u_{w},v_{j_{1}}),(u_{w},v_{j_{2}}),\cdots,(u_{w},v_{j_{h}}), y$ to $x$ in $H(u_{w})$, where $\overrightarrow{R_{s}^{(u_{w})}}$ is the directed $(u_{w},v_{j_{1}})-x$ path for $s\in[h]$, and $\overrightarrow{R_{h+1}^{(u_{w})}}$ is the directed $y-x$ path. For each $s\in [h]$, we construct a tree $T_{s}'$ as follows:
            \begin{itemize}
                \item $V(T_{s}')=%V(\overrightarrow{Q_{v_{a}v_{2},s}^{(u_{p})}[v_{a},v_{j_{s}}]})
                V(\overrightarrow{Q_{v_{a}v_{2},s}^{(u_{p})}[v_{a},v_{j_{s}}]})\cup V(\overrightarrow{P_{u_{p}u_{w}}^{(v_{j_{s}})}})\cup V(\overrightarrow{R_{s}^{(u_{w})}})\cup \{y\}$;
                \item $A(T_{s}')=%A(\overrightarrow{Q_{v_{a}v_{2},s}^{(u_{p})}[v_{a},v_{j_{s}}]})
                A(\overrightarrow{Q_{v_{a}v_{2},s}^{(u_{p})}[v_{a},v_{j_{s}}]})\cup A(\overrightarrow{P_{u_{p}u_{w}}^{(v_{j_{s}})}})\cup A(\overrightarrow{R_{s}^{(u_{w})}})\cup \{(u_{w},v_{j_{s}})y\}$.
            \end{itemize}
            \end{SubProofCase}
        \end{SubProofCase}
        It can be verified that across both subcases, these are $h$ pairwise internally-disjoint pendant $(S,r)$-trees.
    \end{ProofCase}
    Based on the above arguments, it can be conclude that for any $S=\{r,x,y\}\subseteq V(D\square H)$ with $S_{D}=\{u_{p},u_{w}\}$ and the projection of $r$ onto $D$ is exactly $u_{p}$, there exist $\tau_{3}(H)$ pairwise internally-disjoint pendant $(S,r)$-tree in $\overrightarrow{P_{u_{p}u_{w}}}\square H$. %Hence, we have $\tau_{3}(\overrightarrow{P_{u_{p}u_{w}}}\square H)\geq \tau_{3}(H)$.
\end{proof}

%From the proof above, we immediately obtain the following corollary.
%\begin{corollary}
    %Let $D$ and $H$ be two strong digraphs, and let $S=\{r,x,y\}\subseteq V(D\square H)$. If $S_{D}=\{u_{p},u_{w}\}\subseteq V(D)$ with $p\neq w$, then we have
    %\begin{align*}
        %\tau_{3}(\overleftrightarrow{P_{u_{p}u_{w}}}\square H)\geq \tau_{3}(H).
    %\end{align*}
    %Moreover, the bound is sharp.
%\end{corollary}

Next, we give the proof of Theorem \ref{thm1.1}. For any two strong digraphs $D$ and $H$, we need to show that
\begin{align*}
    \tau_{3}(D\square H)\geq \tau_{3}(D)+\tau_{3}(H).
    %\tau_{3}(D\square H)\geq \min \{3\lfloor\frac{\tau_{3}(D)}{2}\rfloor,3\lfloor\frac{\tau_{3}(H)}{2}\rfloor\}.
\end{align*}
By the symmetry of Cartesian product digraphs, we may assume $\tau_{3}(H)\geq \tau_{3}(D)\geq 1$. Let $\tau_{3}(D)=\ell$ and $\tau_{3}(H)=h$. From the definition of $\tau_{3}(D\square H)$, it suffices to show that $\tau_{S,r}(D\square H)\geq \ell+h$ for any $S=\{r,x,y\}\subseteq V(D\square H)$, i.e., we need to find $\ell+h$ pairwise internally-disjoint pendant $(S,r)$-trees in $D\square H$.
%By the symmetry of Cartesian product digraphs, we may assume $\tau_{3}(H)\geq \tau_{3}(D)\geq 1$. Hence, we only need to show that $\tau_{3}(D\square H)\geq 3\lfloor\frac{\tau_{3}(D)}{2}\rfloor$. Let $\tau_{3}(D)=\ell$ and $\tau_{3}(H)=h$. From the definition of $\tau_{3}(D\square H)$, it suffices to show that $\tau_{S}(D\square H)\geq 3\lfloor\frac{\ell}{2}\rfloor$ for any $S=\{r,x,y\}\subseteq V(D\square H)$ and $|S|=3$, i.e., we need to find $3\lfloor\frac{\ell}{2}\rfloor$ pairwise internally-disjoint pendant $(S,r)$-trees in $D\square H$. 
%Let $r=(u_{p},v_{a})$, $x=(u_{q},v_{b})$ and $y=(u_{w},v_{c})$, thus we use $S_{D}=\{u_{p},u_{q},u_{r}\}$ to denote the {\em projection} of $S$ onto $D(v_{a})$. 
We proceed our proof by the following three lemmas.

\begin{lemma}\label{lem3.1}
    If all vertices of $S$ lie in the same $H(u_{i})~(i\in [n])$, then there exist $\ell+h+1$ pairwise internally-disjoint pendant $(S,r)$-trees in $D\square H$.
\end{lemma}

\begin{proof}
    Without loss of generality, let $S=\{r,x,y\}\subseteq V(H(u_{p}))$ with $r=(u_{p},v_{1})$, $x=(u_{p},v_{2})$ and $y=(u_{p},v_{3})$. On the one hand, since $H(u_{p})\cong H$, we have $\tau_{3}(H(u_{p}))=h$. Thus, there exist $h$ pairwise internally-disjoint pendant $(S,r)$-trees in $H(u_{p})$, denoted by $T_{1}',T_{2}',\cdots,T_{h}'$.
    
    On the other hand, since $D(v_{1})\cong D$, we have $\delta^{0}(D(v_{1}))\geq \tau_{3}(D(v_{1}))+2=\ell+2$ by Lemma \ref{lem2.2}. It means that $r$ has at least $\ell+2$ out-neighbors in $D(v_{1})$, say $(u_{i_{1}},v_{1}),(u_{i_{2}},v_{1}),\cdots,(u_{i_{\ell+2}},v_{1})$. In the sequel, we only use $\ell+1$ of these out-neighbors. For each $t\in[\ell+1]$, let $S_{t}=\{(u_{i_{t}},v_{1}),(u_{i_{t}},v_{2}),(u_{i_{t}},v_{3})\}$, and let $T_{1}^{(u_{i_{t}})}$ be a pendant $(S_{t},(u_{i_{t}},v_{1}))$-tree in $H(u_{i_{t}})$, which corresponds to $T_{1}'$ in $H(u_{p})$. By Lemmas~\ref{lem2.1} and~\ref{lem2.3}, $D(v_{2})$ is an $(\ell+1)$-strong digraph and hence %$\kappa(D(v_{2}))\geq\ell+1$, so $D(v_{2})$ is an $(\ell+1)$-strong digraph. By Lemma \ref{lem2.3}, $D(v_{2})$ 
    contains an $(\ell+1)$-fan consisting of $\ell+1$ pairwise internally-disjoint directed paths, denoted by $\overrightarrow{R_{t}^{(v_{2})}}$ for $t\in [\ell+1]$, from $(u_{i_{1}},v_{2}),(u_{i_{2}},v_{2}),\cdots,(u_{i_{\ell+1}},v_{2})$ to $(u_{p},v_{2})$, %in $D(v_{2})$, 
    where each $\overrightarrow{R_{t}^{(v_{2})}}$ is the directed $(u_{i_{t}},v_{2})-(u_{p},v_{2})$ path. %$\overrightarrow{P_{u_{i_{1}}u_{p}}^{(v_{2})}},\overrightarrow{P_{u_{i_{2}}u_{p}}^{(v_{2})}},\cdots,\overrightarrow{P_{u_{i_{\ell+1}}u_{p}}^{(v_{2})}}$. 
    Similarly, $D(v_{3})$ contains $\ell+1$ pairwise internally-disjoint directed paths $\overrightarrow{R_{t}^{(v_{3})}}$ for $t\in [\ell+1]$. %where each $\overrightarrow{P_{u_{i_{t}}u_{p}}^{(v_{3})}}$ from $(u_{i_{1}},v_{3}),(u_{i_{2}},v_{3}),\cdots,(u_{i_{\ell+1}},v_{3})$ to $(u_{p},v_{3})$ in $D(v_{3})$, respectively. 
    We thus define, for each $t\in [\ell+1]$, a tree $T_{t}^{*}$, such that
    \begin{itemize}
        \item $V(T_{t}^{*})=\{r\}\cup V(T_{1}^{(u_{i_{t}})})\cup V(\overrightarrow{R_{t}^{(v_{2})}})\cup V(\overrightarrow{R_{t}^{(v_{3})}})$;
        \item $A(T_{t}^{*})=\{r(u_{i_{t}},v_{1})\}\cup A(T_{1}^{(u_{i_{t}})})\cup A(\overrightarrow{R_{t}^{(v_{2})}})\cup A(\overrightarrow{R_{t}^{(v_{3})}})$.
    \end{itemize}
    One can verify that $T_{1}',T_{2}',\cdots,T_{h}'$ and $T_{1}^{*},T_{2}^{*},\cdots,T_{\ell+1}^{*}$ form $\ell+h+1$ pairwise internally-disjoint pendant $(S,r)$-trees in $D\square H$. %Hence, we obtain
    %\begin{align*}
        %\tau_{S,r}(D\square H)\geq h+\ell+1\geq 3\lfloor\frac{\ell}{2}\rfloor,
    %\end{align*}
    %we have used the earlier assumption that $h\geq \ell\geq 1$ in the last inequality. 
    This Lemma is now proved.
\end{proof}

\begin{lemma}\label{lem3.2}
    If exactly two vertices of $S$ lie in the same $H(u_{i})~(i\in [n])$, then there exist $\ell+h$ pairwise internally-disjoint pendant $(S,r)$-trees in $D\square H$.
\end{lemma}

\begin{proof}
    Let $S=\{r,x,y\}\subseteq V(D\square H)$, and the argument is divided into the following two cases.
    \begin{ProofCase}
        Only $r$ and $x$ belong to the same $H(u_{i})$. 
        The argument for the case that $r$ and $y$ belong to the same $H(u_{i})$ is similar and so we omit the details.
        
        Without loss of generality, let $r=(u_{p},v_{1})$, $x=(u_{p},v_{2})$ and $y=(u_{w},v_{c})$, where $u_{p},u_{w}$ are distinct. %On the one hand, On the other hand, 
        By Lemma \ref{lem2.1} and Theorem \ref{thm2.1}, %we have $\kappa(D(v_{1}))\geq \ell+1$, which implies that 
        there exist $\ell+1$ pairwise internally-disjoint directed paths from $r$ to $(u_{w},v_{1})$ in $D(v_{1})$, denoted by $\overrightarrow{P_{u_{p}u_{w},1}^{(v_{1})}},\overrightarrow{P_{u_{p}u_{w},2}^{(v_{1})}},\cdots,\overrightarrow{P_{u_{p}u_{w},\ell+1}^{(v_{1})}}$. Among these $\ell+1$ directed paths, at most one consists of only a single arc and we denote this directed path as $\overrightarrow{P_{u_{p}u_{w},\ell+1}^{(v_{1})}}=(u_{p},v_{1})(u_{w},v_{1})$. All the remaining directed paths contain internal vertices, where $(u_{i_{t}},v_{1})$ denotes the in-neighbor of $(u_{w},v_{1})$ in each of these paths for $t\in [\ell]$. %Furthermore, let $\overrightarrow{P_{u_{p}u_{i_{t}}}^{(v_{1})}}=\overrightarrow{P_{u_{p}u_{w},1}^{(v_{1})}}-(u_{i_{t}},v_{1})(u_{w},v_{1})$. %Let $y'=(u_{p},v_{c})$ be the vertex corresponding to $y$ in $H(u_{p})$, $r'=(u_{w},v_{1}),x'=(u_{w},v_{2})$ be the vertices corresponding to $r,x$ in $H(u_{w})$, $r_{i},x_{i},y_{i}$ be the vertices corresponding to $r,x,y$ in $H(u_{q_{i}})$.
        %We first consider the Cartesian product of $\overrightarrow{P_{u_{p}u_{w},\ell+1}^{(v_{1})}}\square H$. 
        By Proposition \ref{pro3.1}, there exist $h$ pairwise internally-disjoint pendant $(S,r)$-trees in $\overrightarrow{P_{u_{p}u_{w},\ell+1}^{(v_{1})}}\square H$, denoted by $T_{1}',T_{2}',\cdots,T_{h}'$.
        
        We next construct $\ell$ additional pendant $(S,r)$-trees using arcs from $D(v_{2})$ and from $(\bigcup_{t=1}^{\ell}\overrightarrow{P_{u_{p}u_{w},t}^{(v_{1})}})\square H$. Moreover, these $\ell$ trees are pairwise internally-disjoint from each other and from the previously constructed $T_{1}',T_{2}',\cdots,T_{h}'$. %to construct $\ell$ internally-disjoint pendant $(S,r)$-trees, whose are also internally-disjoint from $T_{1},T_{2}\cdots,T_{h}$.
        %we consider the Cartesian product of $(\bigcup_{t=1}^{\ell}\overrightarrow{P_{u_{p}u_{w},t}^{(v_{1})}})\square H$.
        \begin{SubProofCase}
            $(u_{i_{t}},v_{c})\notin \{(u_{i_{t}},v_{1}),(u_{i_{t}},v_{2})\}$. For each $t\in [\ell]$, let $S_{t}=\{(u_{i_{t}},v_{1}),(u_{i_{t}},v_{2}),(u_{i_{t}},v_{c})\}$ be a terminal vertex set in $H(u_{i_{t}})$. Since $H(u_{i_{t}})\cong H$, there exists a pendant $(S_{t},(u_{i_{t}},v_{1}))$-tree $T_{1}^{(u_{i_{t}})}$ in $H(u_{i_{t}})$, which corresponds to $T_{1}^{(u_{p})}$ as denoted in Case 1 of Proposition \ref{pro3.1} for $H(u_{p})$. %Meanwhile, let $(u_{i_{t}},v_{j_{1}})$ be the in-neighbor of $(u_{i_{t}},v_{c})$ in $T_{1}^{(u_{i_{t}})}$.
            %By Lemma \ref{lem2.1}, $D(v_{1})$ is an $(\ell+1)$-strong digraph since $\kappa(D(v_{1}))\geq \ell+1$. It can be inferred that there exist $\ell+1$ pairwise internally-disjoint directed paths from $(u_{i_{1}},v_{2}),(u_{i_{2}},v_{2}),\cdots,(u_{i_{\ell}},v_{2}),(u_{w},v_{2})$ to $(u_{p},v_{2})$, respectively, denoted by $\overrightarrow{P}_{1},\overrightarrow{P}_{2},\cdots,\overrightarrow{P}_{\ell+1}$. 
            Since $\kappa(D(v_{2}))\geq \ell+1$, by Lemma \ref{lem2.3}, there exist $\ell$ pairwise internally-disjoint directed paths from $(u_{i_{1}},v_{2}),(u_{i_{2}},v_{2}),\cdots,(u_{i_{\ell}},v_{2})$ to $(u_{p},v_{2})$, respectively. These paths are denoted by $\overrightarrow{R_{1}^{(v_{2})}},\overrightarrow{R_{2}^{(v_{2})}},\cdots,\overrightarrow{R_{\ell}^{(v_{2})}}$. For each $t\in[\ell]$, we define a tree $T_{t}^{*}$ as follows:
            \begin{itemize}
                \item $V(T_{t}^{*})=V(\overrightarrow{P_{u_{p}u_{w},t}^{(v_{1})}[u_{p},u_{i_{t}}]}) \cup V(T_{1}^{(u_{i_{t}})})\cup V(\overrightarrow{R_{t}^{(v_{2})}}) \cup \{y\}$;
                \item $A(T_{t}^{*})=A(\overrightarrow{P_{u_{p}u_{w},t}^{(v_{1})}[u_{p},u_{i_{t}}]}) \cup A(T_{1}^{(u_{i_{t}})}) \cup A(\overrightarrow{R_{t}^{(v_{2})}}) \cup \{(u_{i_{t}},v_{c})y\}$.
            \end{itemize}
            It can be checked that these are $\ell$ pairwise internally-disjoint pendant $(S,r)$-trees.%, and that they are also internally-disjoint from the previously constructed $T_{1}',T_{2}',\cdots,T_{h}'$.
        \end{SubProofCase}

        \begin{SubProofCase}
            $(u_{i_{t}},v_{c})=(u_{i_{t}},v_{2})$. For each $t\in [\ell]$, let $\overrightarrow{Q_{v_{1}v_{2},1}^{(u_{i_{t}})}}$ be a directed path from $(u_{i_{t}},v_{1})$ to $(u_{i_{t}},v_{2})$ in $H(u_{i_{t}})$, which corresponds to $\overrightarrow{Q_{rx,1}}$ as denoted in Subcase 1.2 of Proposition \ref{pro3.1} for $H(u_{p})$.
            %By Lemma \ref{lem2.1}, we have $\kappa(H)\geq h+1$, which implies that $H(u_{q_{i}})$ is an $(h+1)$-strong and there always exist an internally-disjoint paths from $(u_{q_{i}},v_{1})$ to $(u_{q_{i}},v_{2})$, respectively, denoted by $\overrightarrow{R}_{1},\overrightarrow{R}_{2},\cdots,\overrightarrow{R}_{\ell}$. 
            Since $D(v_{2})$ is an $(\ell+1)$-strong digraph, there exist $\ell+1$ pairwise internally-disjoint directed paths from $(u_{i_{1}},v_{2}),(u_{i_{2}},v_{2}),\cdots,(u_{i_{\ell}},v_{2}),y$ to $(u_{p},v_{2})$, respectively. These paths are denoted by $\overrightarrow{R_{1}^{(v_{2})}},\overrightarrow{R_{2}^{(v_{2})}},\cdots,\overrightarrow{R_{\ell+1}^{(v_{2})}}$. For each $t\in [\ell]$, we define a tree $T_{t}^{*}$ such that
            \begin{itemize}
                \item $V(T_{t}^{*})=V(\overrightarrow{P_{u_{p}u_{w},t}^{(v_{1})}[u_{p},u_{i_{t}}]}) \cup V(\overrightarrow{Q_{v_{1}v_{2},1}^{(u_{i_{t}})}})\cup V(\overrightarrow{R_{t}^{(v_{2})}}) \cup \{y\}$;
                \item $A(T_{t}^{*})=A(\overrightarrow{P_{u_{p}u_{w},t}^{(v_{1})}[u_{p},u_{i_{t}}]}) \cup A(\overrightarrow{Q_{v_{1}v_{2},1}^{(u_{i_{t}})}})\cup A(\overrightarrow{R_{t}^{(v_{2})}}) \cup \{(u_{i_{t}},v_{c})y\}$.
            \end{itemize}
            It can be checked that these are $\ell$ pairwise internally-disjoint pendant $(S,r)$-trees.%, and that they are also internally-disjoint from the previously constructed $T_{1}',T_{2}',\cdots,T_{h}'$.
        \end{SubProofCase}

        \begin{SubProofCase}
            $(u_{i_{t}},v_{c})=(u_{i_{t}},v_{1})$. By the same argument as in Subcase 1.2, we obtain $\overrightarrow{Q_{v_{1}v_{2},1}^{(u_{i_{1}})}},\overrightarrow{Q_{v_{1}v_{2},1}^{(u_{i_{2}})}},\cdots,\overrightarrow{Q_{v_{1}v_{2},1}^{(u_{i_{\ell}})}}$ and $\overrightarrow{R_{1}^{(v_{2})}},\overrightarrow{R_{2}^{(v_{2})}},\cdots,\overrightarrow{R_{\ell}^{(v_{2})}}$. %$\overrightarrow{R}_{1},\overrightarrow{R}_{2},\cdots,\overrightarrow{R}_{\ell}$ and $\overrightarrow{P}_{1},\overrightarrow{P}_{2},\cdots,\overrightarrow{P}_{\ell+1}$. Without loss of generality, let $\overrightarrow{P}_{1},\overrightarrow{P}_{2},\cdots,\overrightarrow{P}_{\ell}$ are contained at least one internal vertex, and we denote $(u_{q_{i}},v_{g_{i}})$ be the out-neighbor of $(u_{q_{i}},v_{1})$ in $\overrightarrow{P}_{i}$, respectively. Since $H(u_{w})$ is an $(h+1)$-strong digraph and there exist $h+1$ pairwise internally-disjoint directed paths from $(u_{w},v_{g_{1}}),(u_{w},v_{g_{2}}),\cdots,(u_{w},v_{g_{\ell}}),(u_{w},v_{2})$ to $(u_{w},v_{1})$, respectively, denoted by $\overrightarrow{Q}_{1},\overrightarrow{Q}_{2},\cdots,\overrightarrow{Q}_{\ell+1}$. 
            For each $t\in [\ell]$, we define a tree $T_{t}^{*}$ such that
            \begin{itemize}
                \item $V(T_{t}^{*})=V(\overrightarrow{P_{u_{p}u_{w},t}^{(v_{1})}})\cup V(\overrightarrow{Q_{v_{1}v_{2},1}^{(u_{i_{t}})}})\cup V(\overrightarrow{R_{t}^{(v_{2})}})$;
                \item $A(T_{t}^{*})=A(\overrightarrow{P_{u_{p}u_{w},t}^{(v_{1})}})\cup A(\overrightarrow{Q_{v_{1}v_{2},1}^{(u_{i_{t}})}})\cup A(\overrightarrow{R_{t}^{(v_{2})}})$.
            \end{itemize}
            It can be checked that these are $\ell$ pairwise internally-disjoint pendant $(S,r)$-trees. 
        \end{SubProofCase}
    \end{ProofCase}

    \begin{ProofCase}
        Only $x$ and $y$ belong to the same $H(u_{i})$. Without loss of generality, let $r=(u_{p},v_{a})$, $x=(u_{w},v_{1})$ and $y=(u_{w},v_{2})$. Analogous to Case 1, we obtain $\ell+1$ directed $u_{p}-u_{w}$ paths in $D(v_{a})$, denoted by $\overrightarrow{P_{u_{p}u_{w},1}^{(v_{a})}},\overrightarrow{P_{u_{p}u_{w},2}^{(v_{a})}},\cdots,\overrightarrow{P_{u_{p}u_{w},\ell+1}^{(v_{a})}}$. Let $\overrightarrow{P_{u_{p}u_{w},\ell+1}^{(v_{a})}}=(u_{p},v_{a})(u_{w},v_{a})$, and for each $t\in[\ell]$, let $(u_{i_{t}},v_{a})$ be the in-neighbor of $(u_{w},v_{a})$ in $\overrightarrow{P_{u_{p}u_{w},t}^{(v_{a})}}$. %We further define $\overrightarrow{P_{u_{p}u_{i_{t}}}^{(v_{a})}}=\overrightarrow{P_{u_{p}u_{w},t}^{(v_{a})}}-(u_{i_{t}},v_{a})(u_{w},v_{a})$. 
        By Proposition \ref{pro3.1}, there exist $h$ pairwise internally-disjoint pendant $(S,r)$-trees in $\overrightarrow{P_{u_{p}u_{w},\ell+1}^{(v_{a})}}\square H$, denoted by $T_{1}',T_{2}',\cdots,T_{h}'$. 
        \begin{SubProofCase}
            $(u_{i_{t}},v_{a})\notin \{(u_{i_{t}},v_{1}),(u_{i_{t}},v_{2})\}$. As in Subcase 1.1, for each $t\in [\ell]$, let $S_{t}=\{(u_{i_{t}},v_{a}),(u_{i_{t}},v_{1}),(u_{i_{t}},v_{2})\}$ and $T_{1}^{(u_{i_{t}})}$ be a pendant $(S_{t},(u_{i_{t}},v_{a}))$-tree in $H(u_{i_{t}})$, which corresponds to $T_{1}^{(u_{p})}$ as denoted in Subcase 1.1 of Proposition \ref{pro3.1} for $H(u_{p})$. %(this tree always exists since $\tau_{3}(H(u_{i}))=h\geq \ell$.) 
            We further define a tree $T_{t}^{*}$ as follows:
            \begin{itemize}
                \item $V(T_{t}^{*})=V(\overrightarrow{P_{u_{p}u_{w},t}^{(v_{a})}[u_{p},u_{i_{t}}]})\cup V(T_{1}^{(u_{i_{t}})})\cup \{x,y\}$;
                \item $A(T_{t}^{*})=A(\overrightarrow{P_{u_{p}u_{w},t}^{(v_{a})}[u_{p},u_{i_{t}}]})\cup A(T_{1}^{(u_{i_{t}})})\cup \{(u_{i_{t}},v_{1})x,(u_{i_{t}},v_{2})y\}$.
            \end{itemize}
            It can be verified that these are $\ell$ pairwise internally-disjoint pendant $(S,r)$-trees.
        \end{SubProofCase}

        \begin{SubProofCase}
            $(u_{i_{t}},v_{a})=(u_{i_{t}},v_{1})$. The argument for the case where $(u_{i_{t}},v_{a})=(u_{i_{t}},v_{2})$ is similar, and we thus omit the details. %By the same reason as in Subcase 1.2, for each $t\in[\ell]$, there exists a directed $v_{a}-v_{2}$ path $\overrightarrow{Q_{v_{a}v_{2},1}^{(u_{i_{t}})}}$ in $H(u_{i_{t}})$. 
            Analogous to Subcase 1.2, we obtain $\overrightarrow{Q_{v_{1}v_{2},1}^{(u_{i_{1}})}},\overrightarrow{Q_{v_{1}v_{2},1}^{(u_{i_{2}})}},\cdots,\overrightarrow{Q_{v_{1}v_{2},1}^{(u_{i_{\ell}})}}$. For each $t\in[\ell]$, we define a tree $T_{t}^{*}$ as follows:
            \begin{itemize}
                \item $V(T_{t}^{*})=V(\overrightarrow{P_{u_{p}u_{w},t}^{(v_{a})}})\cup V(\overrightarrow{Q_{v_{1}v_{2},1}^{(u_{i_{t}})}})\cup \{y\}$;
                \item $A(T_{t}^{*})=A(\overrightarrow{P_{u_{p}u_{w},t}^{(v_{a})}})\cup A(\overrightarrow{Q_{v_{1}v_{2},1}^{(u_{i_{t}})}})\cup \{(u_{i_{t}},v_{2})y\}$.
            \end{itemize}
            %For each $t\in[\ell]$, let $\overrightarrow{Q_{v_{a}v_{2}}^{(u_{i_{t}})}}$ be a directed $(v_{a}-v_{2})$ path in $H(u_{i_{t}})$, which is corresponding to $\overrightarrow{Q_{rx}^{(u_{i_{t}})}}$ in $H(u_{i_{t}})$.
            %Since $\kappa(H(u_{i_{t}}))\geq \ell+1$, we can obtain that there always exist one directed paths from $(u_{q_{i}},v_{a})$ to $(u_{q_{i}},v_{2})$, respectively, denoted by $\overrightarrow{P}_{1},\overrightarrow{P}_{2},\cdots,\overrightarrow{P}_{\ell+1}$. Note that the length of the paths $\overrightarrow{P}_{1},\overrightarrow{P}_{2},\cdots,\overrightarrow{P}_{\ell}$ is at least two, and we denote the in-neighbor of $(u_{q_{i}},v_{2})$ is $(u_{q_{i}},v_{f_{i}})$. Since $H(u_{q_{i}})$ is an $(\ell+1)$-strong digraph, it follows from Lemma \ref{lem2.3} that there exist $\ell+1$ directed paths from $(u_{w},v_{f_{1}}),(u_{w},v_{f_{2}}),\cdots,(u_{w},v_{f_{\ell}}),(u_{w},v_{a})$ to $(u_{w},v_{1})$, respectively, denoted by $\overrightarrow{R}_{1},\overrightarrow{R}_{2},\cdots,\overrightarrow{R}_{\ell}$. For each $i\in [\ell]$, let $T_{i}^{*}$ be a pendant $(S,r)$-tree with vertex set $V(T_{i}^{*})=V(\overrightarrow{P}_{i})\cup V(\overrightarrow{P}_{pw}^{(i)})$ and arc set $A(T_{i}^{*})=A(\overrightarrow{P}_{i})\cup A(\overrightarrow{P}_{pw}^{(i)})$. 
            It can be verified that these are $\ell$ pairwise internally-disjoint pendant $(S,r)$-trees.
        \end{SubProofCase}
    \end{ProofCase}
    One can verify that $T_{1}',T_{2}',\cdots,T_{h}'$ and $T_{1}^{*},T_{2}^{*},\cdots,T_{\ell}^{*}$ form $\ell+h$ pairwise internally-disjoint pendant $(S,r)$-trees in $D\square H$. %Hence, we obtain
    %\begin{align*}
        %\tau_{S,r}(D\square H)\geq h+\ell\geq 3\lfloor\frac{\ell}{2}\rfloor,
    %\end{align*}
    %where we use the earlier assumption that $h\geq \ell\geq 1$ for the last inequality. 
    This Lemma is thus proved.
\end{proof}

\begin{lemma}\label{lem3.3}
    If all vertices of $S$ lie in distinct $H({u_{i}})$~($i\in [n]$), then there exist $\ell+h$ pairwise internally-disjoint pendant $(S,r)$-trees in $D\square H$.
\end{lemma}

\begin{proof}
    Without loss of generality, let $r=(u_{p},v_{a})$, $x=(u_{q},v_{b})$ and $y=(u_{w},v_{c})$, where $u_{p}$, $u_{q}$ and $u_{w}$ are distinct vertices in $D$. We thus define $S_{D}=\{u_{p},u_{q},u_{w}\}$. Since $\tau_{3}(D)=\ell$, there exist $\ell$ pairwise internally-disjoint pendant $(S_{D},u_{p})$-trees in $D$, denoted by $\widetilde{T}_{1},\widetilde{T}_{2},\cdots,\widetilde{T}_{\ell}$. For each $t\in [\ell]$, let $u_{i_{t}}$ be the branch vertex of $\widetilde{T}_{t}$.
    %$\overrightarrow{P_{u_{p}u_{q},t}}$ and $\overrightarrow{P_{u_{p}u_{w},t}}$ be the unique directed $u_{p}-u_{q}$ path and directed $u_{p}-u_{w}$ path in $\widetilde{T}_{t}$, respectively. We further denote by $u_{i_{t}}$ the {\em branch vertex} of $\widetilde{T}_{t}$, i.e., the unique vertex satisfying $V(\overrightarrow{P_{u_{p}u_{q},t}[u_{i_{t}},u_{q}]})\cap V(\overrightarrow{P_{u_{p}u_{w},t}[u_{i_{t}},u_{w}]})=\{u_{i_{t}}\}$. %where $\overrightarrow{P_{u_{i_{t}}u_{q}}}\subseteq \overrightarrow{P_{u_{p}u_{q},t}}$ and $\overrightarrow{P_{u_{i_{t}}u_{w}}}\subseteq \overrightarrow{P_{u_{p}u_{w},t}}$. 
    The proof is divided into the following three cases.
    \begin{ProofCase}
        $v_{a},v_{b}$ and $v_{c}$ are three distinct vertices in $H$. Let $S_{H}=\{v_{a},v_{b},v_{c}\}$. Since $\tau_{3}(H)=h$, there exist $h$ pairwise internally-disjoint pendant $(S_{H},v_{a})$-trees in $H$, denoted by $\hat{T}_{1},\hat{T}_{2},\cdots,\hat{T}_{h}$. For each $s\in [h]$, let $v_{j_{s}}$ be the branch vertex of $\hat{T}_{s}$.
        Since $H(u_{q})$ is an $(h+1)$-strong digraph, there exists an $(h+1)$-fan consisting of $h+1$ pairwise internally-disjoint directed paths $\overrightarrow{R_{s}^{(u_{q})}}$ for $s\in [h+1]$, from $(u_{q},v_{j_{1}}),(u_{q},v_{j_{2}}),\cdots,(u_{q},v_{j_{h}}),(u_{q},v_{a})$ to $x$ in $H(u_{q})$, where $\overrightarrow{R_{s}^{(u_{q})}}$ is the directed $(u_{q},v_{j_{s}})-x$ path for $s\in [h]$, and $\overrightarrow{R_{h+1}^{(u_{q})}}$ is the directed $(u_{q},v_{a})-x$ path. Similarly, $H(u_{w})$ contains $h+1$ pairwise internally-disjoint directed paths $\overrightarrow{R_{s}^{(u_{w})}}$ for $s\in [h+1]$, from $(u_{w},v_{j_{1}}),(u_{w},v_{j_{2}}),\cdots,(u_{w},v_{j_{h}}),(u_{w},v_{a})$ to $y$, where $\overrightarrow{R_{s}^{(u_{w})}}$ is the directed $(u_{w},v_{j_{s}})-y$ path for $s\in [h]$, and $\overrightarrow{R_{h+1}^{(u_{w})}}$ is the directed $(u_{w},v_{a})-y$ path.

        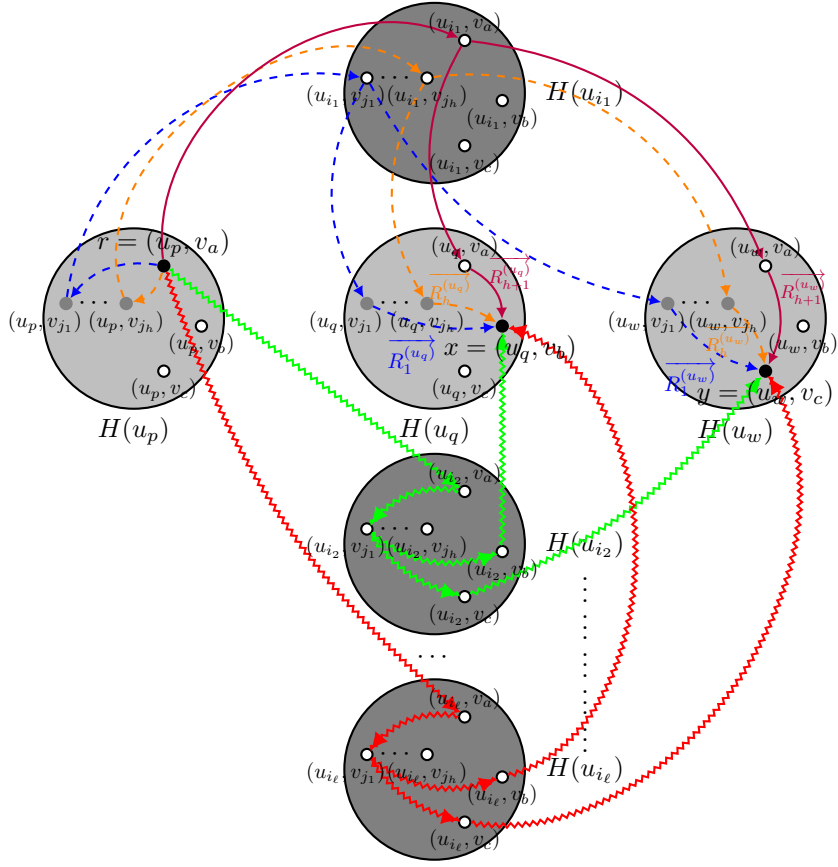
\begin{figure}[htb] 
            \centering
            \begin{tikzpicture}
                % $H(u_{p})$层
                \draw[thick, fill=lightgray] (-4,0) circle (1.2);
                \node[solidvertex] (a11) at (-3.6,0.7) {};
                \node[vertex, gray] (a12) at (-4.9,0.2) {};
                \node at (-4.5,0.2) {$\cdots$};
                \node[vertex, gray] (a13) at (-4.1,0.2) {};
                \node[vertex] (a14) at (-3.1,-0.1) {};
                \node[vertex] (a15) at (-3.6,-0.7) {};
                \node at (-4,-1.5) {$H(u_{p})$};
                % $H(u_{q})$层
                \draw[thick, fill=lightgray] (0,0) circle (1.2);
                \node[vertex] (a21) at (0.4,0.7) {};
                \node[vertex, gray] (a22) at (-0.9,0.2) {};
                \node at (-0.5,0.2) {$\cdots$};
                \node[vertex, gray] (a23) at (-0.1,0.2) {};
                \node[solidvertex] (a24) at (0.9,-0.1) {};
                \node[vertex] (a25) at (0.4,-0.7) {};
                \node at (0,-1.5) {$H(u_{q})$};
                % $H(u_{w})$层
                \draw[thick, fill=lightgray] (4,0) circle (1.2);
                \node[vertex] (a31) at (4.4,0.7) {};
                \node[vertex, gray] (a32) at (3.1,0.2) {};
                \node at (3.5,0.2) {$\cdots$};
                \node[vertex, gray] (a33) at (3.9,0.2) {};
                \node[vertex] (a34) at (4.9,-0.1) {};
                \node[solidvertex] (a35) at (4.4,-0.7) {};
                \node at (4,-1.5) {$H(u_{w})$};
                % $H(u_{i_{1}})$层
                \draw[thick, fill=gray] (0,3) circle (1.2);
                \node[vertex] (a41) at (0.4,3.7) {};
                \node[vertex] (a42) at (-0.9,3.2) {};
                \node at (-0.5,3.2) {$\cdots$};
                \node[vertex] (a43) at (-0.1,3.2) {};
                \node[vertex] (a44) at (0.9,2.9) {};
                \node[vertex] (a45) at (0.4,2.3) {};
                \node at (2,3) {$H(u_{i_{1}})$};
                % $H(u_{i_{2}})$层
                \draw[thick, fill=gray] (0,-3) circle (1.2);
                \node[vertex] (a51) at (0.4,-2.3) {};
                \node[vertex] (a52) at (-0.9,-2.8) {};
                \node at (-0.5,-2.8) {$\cdots$};
                \node[vertex] (a53) at (-0.1,-2.8) {};
                \node[vertex] (a54) at (0.9,-3.1) {};
                \node[vertex] (a55) at (0.4,-3.7) {};
                \node at (0,-4.5) {$\cdots$};
                \node at (2,-3) {$H(u_{i_{2}})$};
                \node at (2,-3.5) {$\vdots$};
                \node at (2,-4) {$\vdots$};
                \node at (2,-4.5) {$\vdots$};
                \node at (2,-5) {$\vdots$};
                \node at (2,-5.5) {$\vdots$};
                % $H(u_{i_{\ell+1}})$层
                \draw[thick, fill=gray] (0,-6) circle (1.2);
                \node[vertex] (a61) at (0.4,-5.3) {};
                \node[vertex] (a62) at (-0.9,-5.8) {};
                \node at (-0.5,-5.8) {$\cdots$};
                \node[vertex] (a63) at (-0.1,-5.8) {};
                \node[vertex] (a64) at (0.9,-6.1) {};
                \node[vertex] (a65) at (0.4,-6.7) {};
                \node at (2,-6) {$H(u_{i_{\ell}})$};
                % 绘制悬挂树$T_{1}$
                \draw[base arc, dashed, bend right=30, blue] (a11) to (a12);
                \draw[base arc, dashed, bend left=45, blue] (a12) to (a42);
                \draw[base arc, dashed, bend right=30, blue] (a42) to (a22);
                \draw[base arc, dashed, bend right=25, blue] (a42) to (a32);
                \draw[base arc, dashed, bend right=15, blue] (a22) to (a24);
                \node at (-0.3, -0.5) [scale=0.8, blue] {$\overrightarrow{R_{1}^{(u_{q})}}$};
                \draw[base arc, dashed, bend right=20, blue] (a32) to (a35);
                \node at (3.4, -0.8) [scale=0.8, blue] {$\overrightarrow{R_{1}^{(u_{w})}}$};
                % 绘制悬挂树$T_{h}$
                \draw[base arc, dashed, bend left=30, orange] (a11) to (a13);
                \draw[base arc, dashed, bend left=75, orange] (a13) to (a43);
                \draw[base arc, dashed, bend right=30, orange] (a43) to (a23);
                \draw[base arc, dashed, bend left=45, orange] (a43) to (a33);
                \draw[base arc, dashed, bend left=15, orange] (a23) to (a24);
                \node at (0.2, 0.4) [scale=0.7, orange] {$\overrightarrow{R_{h}^{(u_{q})}}$};
                \draw[base arc, dashed, bend left=20, orange] (a33) to (a35);
                \node at (3.9, -0.3) [scale=0.7, orange] {$\overrightarrow{R_{h}^{(u_{w})}}$};
                % 绘制悬挂树$T_{1}^{*}$
                \draw[zigzag, green] (a11) to (a51);
                \draw[zigzag, bend right=30, green] (a51) to (a52);
                \draw[zigzag, bend right=30, green] (a52) to (a54);
                \draw[zigzag, bend right=20, green] (a52) to (a55);
                \draw[zigzag, green] (a54) to (a24);
                \draw[zigzag, bend right=15, green] (a55) to (a35);
                % 绘制悬挂树$T_{\ell-1}^{*}$
                \draw[zigzag, bend right=15, red] (a11) to (a61);
                \draw[zigzag, bend right=30, red] (a61) to (a62);
                \draw[zigzag, bend right=30, red] (a62) to (a64);
                \draw[zigzag, bend right=20, red] (a62) to (a65);
                \draw[zigzag, bend right=75, red] (a64) to (a24);
                \draw[zigzag, bend right=65, red] (a65) to (a35);
                % 绘制悬挂树$T_{\ell}^{*}$
                \draw[base arc, bend left=60, purple] (a11) to (a41);
                \draw[base arc, bend right=30, purple] (a41) to (a21);
                \draw[base arc, bend left=30, purple] (a41) to (a31);
                \draw[base arc, bend left=30, purple] (a21) to (a24);
                \node at (1, 0.6) [scale=0.7, purple] {$\overrightarrow{R_{h+1}^{(u_{q})}}$};
                \draw[base arc, bend left=25, purple] (a31) to (a35);
                \node at (4.9, 0.4) [scale=0.7, purple] {$\overrightarrow{R_{h+1}^{(u_{w})}}$};
                % 标记顶点
                \labelNode{above}{a11}{r=(u_{p},v_{a})}
                \labelNode{below, scale=0.8, xshift=-10pt}{a12}{(u_{p},v_{j_{1}})}
                \labelNode{below, scale=0.8}{a13}{(u_{p},v_{j_{h}})}
                \labelNode{below, scale=0.8}{a14}{(u_{p},v_{b})}
                \labelNode{below, scale=0.8}{a15}{(u_{p},v_{c})}
                \labelNode{above, scale=0.8}{a21}{(u_{q},v_{a})}
                \labelNode{below, scale=0.8, xshift=-10pt}{a22}{(u_{q},v_{j_{1}})}
                \labelNode{below, scale=0.8}{a23}{(u_{q},v_{j_{h}})}
                \labelNode{below, xshift=3pt}{a24}{x=(u_{q},v_{b})}
                \labelNode{below, scale=0.8}{a25}{(u_{q},v_{c})}
                \labelNode{above, scale=0.8}{a31}{(u_{w},v_{a})}
                \labelNode{below, scale=0.8, xshift=-10pt}{a32}{(u_{w},v_{j_{1}})}
                \labelNode{below, scale=0.8}{a33}{(u_{w},v_{j_{h}})}
                \labelNode{below, scale=0.8}{a34}{(u_{w},v_{b})}
                \labelNode{below}{a35}{y=(u_{w},v_{c})}
                \labelNode{above, scale=0.8}{a41}{(u_{i_{1}},v_{a})}
                \labelNode{below, scale=0.8, xshift=-10pt}{a42}{(u_{i_{1}},v_{j_{1}})}
                \labelNode{below, scale=0.8}{a43}{(u_{i_{1}},v_{j_{h}})}
                \labelNode{below,, scale=0.8}{a44}{(u_{i_{1}},v_{b})}
                \labelNode{below,, scale=0.8}{a45}{(u_{i_{1}},v_{c})}
                \labelNode{above, scale=0.8}{a51}{(u_{i_{2}},v_{a})}
                \labelNode{below, scale=0.8, xshift=-10pt}{a52}{(u_{i_{2}},v_{j_{1}})}
                \labelNode{below, scale=0.8}{a53}{(u_{i_{2}},v_{j_{h}})}
                \labelNode{below,, scale=0.8}{a54}{(u_{i_{2}},v_{b})}
                \labelNode{below,, scale=0.8}{a55}{(u_{i_{2}},v_{c})}
                \labelNode{above, scale=0.8}{a61}{(u_{i_{\ell}},v_{a})}
                \labelNode{below, scale=0.8, xshift=-10pt}{a62}{(u_{i_{\ell}},v_{j_{1}})}
                \labelNode{below, scale=0.8}{a63}{(u_{i_{\ell}},v_{j_{h}})}
                \labelNode{below,, scale=0.8}{a64}{(u_{i_{\ell}},v_{b})}
                \labelNode{below,, scale=0.8}{a65}{(u_{i_{\ell}},v_{c})}
                % 图的注释
                %\node at (-4.5, -3) {Annotation:};
                %\node at (-4.5, -3.5) [blue] {Blue dashed arcs: $T_{1}$};
                %\node at (-4.5, -4) [orange] {Orange dashed arcs: $T_{h}$};
                %\node at (-4.5, -4.5) [green] {Green zigzag arcs: $T_{1}^{*}$};
                %\node at (-4.5, -5) [red] {Red zigzag arcs: $T_{\ell-1}^{*}$};
                %\node at (-4.5, -5.5) [purple] {Purple arcs: $T_{\ell}^{*}$};
            \end{tikzpicture}
            \caption{$\ell+h$ pendant $(S,r)$-trees for Case 1 of Lemma \ref{lem3.3}.}
            \label{fig2}
        \end{figure}
        For each $s\in [h]$, let $T_{s}'$ be a pendant $(S,r)$-tree induced by the arcs in
        \begin{align*}
            %A(\overrightarrow{Q_{v_{a}v_{b},s}^{(u_{p})}[v_{a},v_{j_{s}}]})
            A(\hat{T}_{s}^{(u_{p})}[v_{a},v_{j_{s}}])\cup A(\widetilde{T}_{1}^{(v_{j_{s}})})\cup A(\overrightarrow{R_{s}^{(u_{q})}})\cup A(\overrightarrow{R_{s}^{(u_{w})}}),
        \end{align*}
        where, to illustrate the pattern, $T_{1}'$ and $T_{h}'$ are depicted by the blue and orange dashed lines in Figure \ref{fig2}. For each $t\in [\ell-1]$, let $T_{t}^{*}$ be a pendant $(S,r)$-tree induced by the arcs in
        \begin{align*}
            %A(\overrightarrow{P_{u_{p}u_{q},t+1}^{(v_{a})}[u_{p},u_{i_{t+1}}]})
            A(\widetilde{T}_{t+1}^{(v_{a})}[u_{p},u_{i_{t+1}}])\cup A(\hat{T}_{1}^{(u_{i_{t+1}})})\cup A(\widetilde{T}_{t+1}^{(v_{b})}[u_{i_{t+1}},u_{q}])\cup A(\widetilde{T}_{t+1}^{(v_{c})}[u_{i_{t+1}},u_{w}]),
            %A(\overrightarrow{P_{u_{p}u_{q},t+1}^{(v_{b})}[u_{i_{t+1}},u_{q}]})\cup A(\overrightarrow{P_{u_{p}u_{w},t+1}^{(v_{c})}[u_{i_{t+1}},u_{w}]}).
        \end{align*}
        where, similarly, $T_{1}^{*}$ and $T_{\ell-1}^{*}$ are shown by the green and red zigzag lines in Figure \ref{fig2}. Finally, let $T_{\ell}^{*}$ be a pendant $(S,r)$-tree induced by the arcs in
        \begin{align*}
            A(\widetilde{T}_{1}^{(v_{a})})\cup A(\overrightarrow{R_{h+1}^{(u_{q})}})\cup A(\overrightarrow{R_{h+1}^{(u_{w})}}),
        \end{align*}
        where $T_{\ell}^{*}$ is represented by the purple lines in Figure \ref{fig2}.
        %$T_{1}',T_{h}',T_{1}^{*},T_{\ell-1}^{*}$ and $T_{\ell}^{*}$ are shown in Figure \ref{fig2}, where blue dashed lines represent $T_{1}$, orange dashed lines represent $T_{h}$, green zigzag lines represent $T_{1}^{*}$, red zigzag lines represent $T_{\ell-1}^{*}$, and purple lines represent $T_{\ell}^{*}$. 
        
        Note that all lines in the figure represent directed paths. It can be verified that %the three pendant $(S,r)$-trees $T^{1,*},T^{2,*},T^{3,*}$ are 
        $T_{1}',T_{2}',\cdots,T_{h}'$ and $T_{1}^{*},T_{2}^{*},\cdots,T_{\ell}^{*}$ together form $\ell+h$ pairwise internally-disjoint pendant $(S,r)$-trees in $D\square H$.
        %and contain no arcs from $W$. In general, for any $T_{i},T_{s}~(1\leq i\neq s\leq \ell)$ and any $\hat{T}_{j},\hat{T}_{t}~(1\leq j\neq t\leq h)$, $(T_{i}\cup T_{s})\square (\hat{T}_{j}\cup \hat{T}_{t})$ contains three pairwise internally-disjoint pendant $(S,r)$-trees and no arcs from $W$. Note that $(\bigcup_{i=1}^{\ell}T_{i})\square (\bigcup_{j=1}^{h}\hat{T}_{j})$ is a subdigraph of $D\square H$. Since $\ell\leq h$, there exist $3\lfloor\frac{\ell}{2}\rfloor$ pairwise internally-disjoint pendant $(S,r)$-trees in $(\bigcup_{i=1}^{\ell}T_{i})\square (\bigcup_{j=1}^{h}\hat{T}_{j})$, and hence in $D\square H$.
    \end{ProofCase}

    \begin{ProofCase}
        Exactly two of $v_{a},v_{b}$ and $v_{c}$ coincide in $H$.
        \begin{SubProofCase}
            $v_{a}=v_{b}$. The argument for the case that $v_{a}=v_{c}$ is similar, and so we omit the details. Let $S_{H}=\{v_{a},v_{c}\}$. By Lemma \ref{lem2.1} and Theorem \ref{thm2.1}, there exist $h+1$ pairwise internally-disjoint directed paths from $v_{a}$ to $v_{c}$ in $H$, denoted by $\overrightarrow{Q_{v_{a}v_{c},1}},\overrightarrow{Q_{v_{a}v_{c},2}},\cdots,\overrightarrow{Q_{v_{a}v_{c},h+1}}$. Among these directed paths, at most one consists of only a single arc and we denote this directed path as $\overrightarrow{Q_{v_{a}v_{c},h+1}}=v_{a}v_{c}$. All the remaining directed paths contain internal vertices, where $v_{j_{s}}$ denotes the in-neighbor of $v_{c}$ in $\overrightarrow{Q_{v_{a}v_{c},s}}$. By Lemma \ref{lem2.3}, $H(u_{q})$ contains an $(h+1)$-fan consisting of $h+1$ pairwise internally-disjoint directed paths, denoted by $\overrightarrow{R_{s}^{(u_{q})}}$ for $s\in [h+1]$, from $(u_{q},v_{j_{1}}),(u_{q},v_{j_{2}}),\cdots,(u_{q},v_{j_{h}}),(u_{q},v_{c})$ to $(u_{q},v_{b})$, where $\overrightarrow{R_{s}^{(u_{q})}}$ is the directed $(u_{q},v_{j_{1}})-(u_{q},v_{b})$ path for $s\in[h]$, and $\overrightarrow{R_{h+1}^{(u_{q})}}$ is the directed $(u_{q},v_{c})-(u_{q},v_{b})$ path. 
            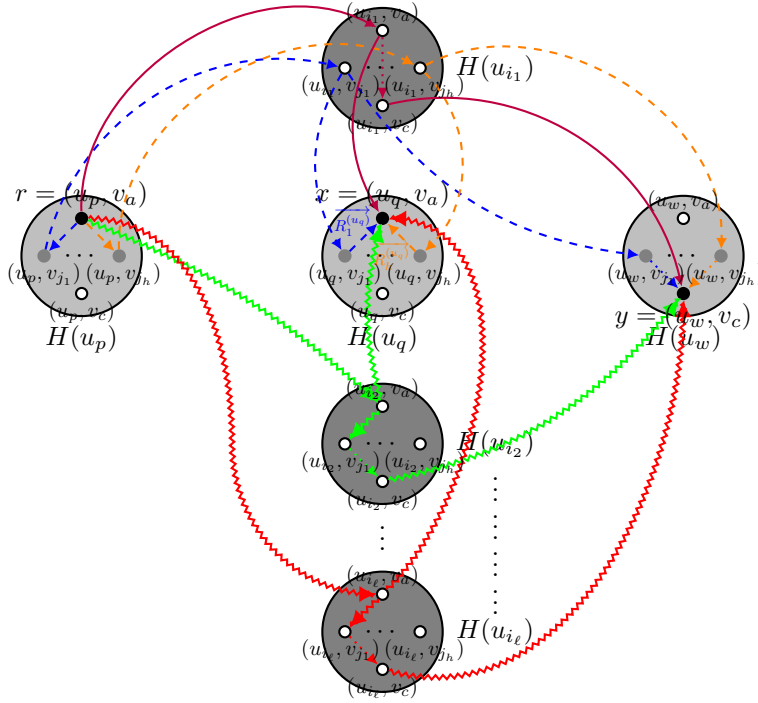
\begin{figure}[htb] 
                \centering
                \begin{tikzpicture}
                    % $H(u_{p})$层
                    \draw[thick, fill=lightgray] (-4,0) circle (0.8);
                    \node[solidvertex] (a11) at (-4,0.5) {};
                    \node[vertex, gray] (a12) at (-4.5,0) {};
                    \node at (-4,0) {$\cdots$};
                    \node[vertex, gray] (a13) at (-3.5,0) {};
                    \node[vertex] (a15) at (-4,-0.5) {};
                    \node at (-4,-1.1) {$H(u_{p})$};
                    % $H(u_{q})$层
                    \draw[thick, fill=lightgray] (0,0) circle (0.8);
                    \node[solidvertex] (a21) at (0,0.5) {};
                    \node[vertex, gray] (a22) at (-0.5,0) {};
                    \node at (0,0) {$\cdots$};
                    \node[vertex, gray] (a23) at (0.5,0) {};
                    \node[vertex] (a25) at (0,-0.5) {};
                    \node at (0,-1.1) {$H(u_{q})$};
                    % $H(u_{w})$层
                    \draw[thick, fill=lightgray] (4,0) circle (0.8);
                    \node[vertex] (a31) at (4,0.5) {};
                    \node[vertex, gray] (a32) at (3.5,0) {};
                    \node at (4,0) {$\cdots$};
                    \node[vertex, gray] (a33) at (4.5,0) {};
                    \node[solidvertex] (a35) at (4,-0.5) {};
                    \node at (4,-1.1) {$H(u_{w})$};
                    % $H(u_{i_{1}})$层
                    \draw[thick, fill=gray] (0,2.5) circle (0.8);
                    \node[vertex] (a41) at (0,3) {};
                    \node[vertex] (a42) at (-0.5,2.5) {};
                    \node at (0,2.5) {$\cdots$};
                    \node[vertex] (a43) at (0.5,2.5) {};
                    \node[vertex] (a45) at (0,2) {};
                    \node at (1.5,2.5) {$H(u_{i_{1}})$};
                    % $H(u_{i_{2}})$层
                    \draw[thick, fill=gray] (0,-2.5) circle (0.8);
                    \node[vertex] (a51) at (0,-2) {};
                    \node[vertex] (a52) at (-0.5,-2.5) {};
                    \node at (0,-2.5) {$\cdots$};
                    \node[vertex] (a53) at (0.5,-2.5) {};
                    \node[vertex] (a55) at (0,-3) {};
                    \node at (0,-3.65) {$\vdots$};
                    \node at (1.5,-2.5) {$H(u_{i_{2}})$};
                    \node at (1.5,-3) {$\vdots$};
                    \node at (1.5,-3.5) {$\vdots$};
                    \node at (1.5,-4) {$\vdots$};
                    \node at (1.5,-4.5) {$\vdots$};
                    % $H(u_{i_{\ell}})$层
                    \draw[thick, fill=gray] (0,-5) circle (0.8);
                    \node[vertex] (a61) at (0,-4.5) {};
                    \node[vertex] (a62) at (-0.5,-5) {};
                    \node at (0,-5) {$\cdots$};
                    \node[vertex] (a63) at (0.5,-5) {};
                    \node[vertex] (a65) at (0,-5.5) {};
                    \node at (1.5,-5) {$H(u_{i_{\ell}})$};
                    % 绘制悬挂树$T_{1}$
                    \draw[base arc, dashed, blue] (a11) to (a12);
                    \draw[base arc, dashed, bend left=45, blue] (a12) to (a42);
                    \draw[base arc, dashed, bend right=30, blue] (a42) to (a22);
                    \draw[base arc, dashed, bend right=25, blue] (a42) to (a32);
                    \draw[base arc, dashed, blue] (a22) to (a21);
                    \node at (-0.4, 0.45) [scale=0.6, blue] {$\overrightarrow{R_{1}^{(u_{q})}}$};
                    \draw[base arc, dotted, blue] (a32) to (a35);
                    % 绘制悬挂树$T_{h}$
                    \draw[base arc, dashed, orange] (a11) to (a13);
                    \draw[base arc, dashed, bend left=60, orange] (a13) to (a43);
                    \draw[base arc, dashed, bend left=45, orange] (a43) to (a23);
                    \draw[base arc, dashed, bend left=60, orange] (a43) to (a33);
                    \draw[base arc, dashed, orange] (a23) to (a21);
                    \node at (0.12, 0) [scale=0.6, orange] {$\overrightarrow{R_{h}^{(u_{q})}}$};
                    \draw[base arc, dotted, orange] (a33) to (a35);
                    % 绘制悬挂树$T_{1}^{*}$
                    \draw[zigzag, bend left=10, green] (a11) to (a51);
                    \draw[zigzag, bend left=15, green] (a51) to (a21);
                    \draw[zigzag, green] (a51) to (a52);
                    \draw[base arc, dotted, green] (a52) to (a55);
                    \draw[zigzag, bend right=20, green] (a55) to (a35);
                    % 绘制悬挂树$T_{\ell-1}^{*}$
                    \draw[zigzag, red] (a11) to[out=0,in=90] (-2,-2.5) to[out=270,in=180] (a61);
                    \draw[zigzag, bend left=105, red] (a61) to[out=-45,in=270]  (a21);
                    \draw[zigzag, red] (a61) to (a62);
                    \draw[base arc, dotted, red] (a62) to (a65);
                    \draw[zigzag, red] (a65) to[out=-15,in=270] (a35);
                    % 绘制悬挂树$T_{\ell}^{*}$
                    \draw[base arc, bend left=60, purple] (a11) to (a41);
                    \draw[base arc, bend right=30, purple] (a41) to (a21);
                    \draw[base arc, dotted, purple] (a41) to (a45);
                    \draw[base arc, bend left=45, purple] (a45) to (a35);
                    % 标记顶点
                    \labelNode{above}{a11}{r=(u_{p},v_{a})}
                    \labelNode{below, scale=0.8}{a12}{(u_{p},v_{j_{1}})}
                    \labelNode{below, scale=0.8, xshift=2pt}{a13}{(u_{p},v_{j_{h}})}
                    \labelNode{below, scale=0.8}{a15}{(u_{p},v_{c})}
                    \labelNode{above}{a21}{x=(u_{q},v_{a})}
                    \labelNode{below, scale=0.8}{a22}{(u_{q},v_{j_{1}})}
                    \labelNode{below, scale=0.8, xshift=2pt}{a23}{(u_{q},v_{j_{h}})}
                    \labelNode{below, scale=0.8}{a25}{(u_{q},v_{c})}
                    \labelNode{above, scale=0.8}{a31}{(u_{w},v_{a})}
                    \labelNode{below, scale=0.8}{a32}{(u_{w},v_{j_{1}})}
                    \labelNode{below, scale=0.8, xshift=2pt}{a33}{(u_{w},v_{j_{h}})}
                    \labelNode{below}{a35}{y=(u_{w},v_{c})}
                    \labelNode{above, scale=0.8}{a41}{(u_{i_{1}},v_{a})}
                    \labelNode{below, scale=0.8, xshift=-2pt}{a42}{(u_{i_{1}},v_{j_{1}})}
                    \labelNode{below, scale=0.8, xshift=2pt}{a43}{(u_{i_{1}},v_{j_{h}})}
                    \labelNode{below,, scale=0.8}{a45}{(u_{i_{1}},v_{c})}
                    \labelNode{above, scale=0.8}{a51}{(u_{i_{2}},v_{a})}
                    \labelNode{below, scale=0.8, xshift=-2pt}{a52}{(u_{i_{2}},v_{j_{1}})}
                    \labelNode{below, scale=0.8, xshift=2pt}{a53}{(u_{i_{2}},v_{j_{h}})}
                    \labelNode{below,, scale=0.8}{a55}{(u_{i_{2}},v_{c})}
                    \labelNode{above, scale=0.8}{a61}{(u_{i_{\ell}},v_{a})}
                    \labelNode{below, scale=0.8, xshift=-2pt}{a62}{(u_{i_{\ell}},v_{j_{1}})}
                    \labelNode{below, scale=0.8, xshift=2pt}{a63}{(u_{i_{\ell}},v_{j_{h}})}
                    \labelNode{below,, scale=0.8}{a65}{(u_{i_{\ell}},v_{c})}
                \end{tikzpicture}
                \caption{$\ell+h$ pendant $(S,r)$-trees for Subcase 2.1 of Lemma \ref{lem3.3}.}
                \label{fig3}
            \end{figure}
            
            For each $s\in [h]$, Let $T_{s}'$ be a pendant $(S,r)$ tree induced by the arcs in
            \begin{align*}
                A(\overrightarrow{Q_{v_{a}v_{c},s}^{(u_{p})}}[v_{a},v_{j_{s}}])\cup A(\widetilde{T}_{1}^{(v_{j_{s}})})\cup A(\overrightarrow{R_{s}^{(u_{q})}})\cup \{(u_{w},v_{j_{s}})y\}.
            \end{align*}
            %\begin{itemize}
                %\item $V(T_{s})=V(\overrightarrow{Q_{v_{a}v_{c},s}^{(u_{p})}[v_{a},v_{j_{s}}]})\cup V(\widetilde{T}_{1}^{(v_{j_{s}})})\cup V(\overrightarrow{R_{s}^{(u_{q})}})\cup \{y\}$;
                %\item $A(T_{s})=A(\overrightarrow{Q_{v_{a}v_{j_{s}},s}^{(u_{p})}}[v_{a},v_{j_{s}}])\cup A(\widetilde{T}_{1}^{(v_{j_{s}})})\cup A(\overrightarrow{R_{s}^{(u_{q})}})\cup \{(u_{w},v_{j_{s}})y\}$.
            %\end{itemize}
            For each $t\in[\ell-1]$, let $T_{t}^{*}$ be a pendant $(S,r)$ tree induced by the arcs in %we define a tree $T_{t}^{*}$ such that
            \begin{align*}
                A(\widetilde{T}_{t+1}^{(v_{a})}[u_{p},u_{q}])\cup A(\overrightarrow{Q_{v_{a}v_{c},1}^{(u_{i_{t+1}})}})\cup A(\widetilde{T}_{t+1}^{(v_{c})}[u_{i_{t+1}},u_{w}]).
            \end{align*}
            %\begin{itemize}
                %\item $V(T_{t}^{*})=V(\overrightarrow{P_{u_{p}u_{q},t+1}^{(v_{a})}})\cup V(\overrightarrow{Q_{v_{a}v_{c},1}^{(u_{i_{t+1}})}})\cup V(\overrightarrow{P_{u_{p}u_{w},t+1}^{(v_{c})}[u_{i_{t+1}},u_{w}]})$;
                %\item $A(T_{t}^{*})=A(\overrightarrow{P_{u_{p}u_{q},t+1}^{(v_{a})}})\cup A(\overrightarrow{Q_{v_{a}v_{c},1}^{(u_{i_{t+1}})}})\cup A(\overrightarrow{P_{u_{p}u_{w},t+1}^{(v_{c})}[u_{i_{t+1}},u_{w}]})$.
            %\end{itemize}
            At last, let $T_{\ell}^{*}$ be a pendant $(S,r)$-tree induced by the arcs in
            \begin{align*}
                A(\widetilde{T}_{1}^{(v_{a})}[u_{p},u_{q}])\cup A(\overrightarrow{Q_{v_{a}v_{c},h+1}^{(u_{i_{1}})}})\cup A(\widetilde{T}_{1}^{(v_{c})}[u_{i_{t+1}},u_{w}]).
            \end{align*}
            %As illustrate in Figure \ref{fig3}, we depict the representative trees $T_{1}',T_{h}',T_{1}^{*},T_{\ell-1}^{*}$ and $T_{\ell}^{*}$, where blue dashed lines represent $T_{1}'$, orange dashed lines represent $T_{h}'$, green zigzag lines represent $T_{1}^{*}$, red zigzag lines represent $T_{\ell-1}^{*}$, and purple lines represent $T_{\ell}^{*}$. Note that dotted lines denote arcs, while all other lines represent directed paths. One can verify that $T_{1}',T_{2}',\cdots,T_{h}'$ and $T_{1}^{*},T_{2}^{*},\cdots,T_{\ell-1}^{*},T_{\ell}^{*}$ together form $h+\ell$ pairwise internally-disjoint pendant $(S,r)$-trees in $D\square H$.
        \end{SubProofCase}

        \begin{SubProofCase}
            $v_{b}=v_{c}$. By the same reason as in Subcase 2.1, we obtain $\overrightarrow{Q_{v_{a}v_{c},1}^{(u_{p})}},\overrightarrow{Q_{v_{a}v_{c},2}^{(u_{p})}},\cdots,\overrightarrow{Q_{v_{a}v_{c},h+1}^{(u_{p})}}$. Let $\overrightarrow{Q_{v_{a}v_{c},h+1}^{(u_{p})}}=(u_{p},v_{a})(u_{p},v_{c})$ and $(u_{p},v_{j_{s}})$ be the in-neighbor of $(u_{p},v_{c})$ in $\overrightarrow{Q_{v_{a}v_{c},s}^{(u_{p})}}$ for $s\in [h]$. 
            \begin{figure}[htb] 
                \centering
                \begin{tikzpicture}
                    % $H(u_{p})$层
                    \draw[thick, fill=lightgray] (-4,0) circle (0.8);
                    \node[solidvertex] (a11) at (-4,0.5) {};
                    \node[vertex, gray] (a12) at (-4.5,0) {};
                    \node at (-4,0) {$\cdots$};
                    \node[vertex, gray] (a13) at (-3.5,0) {};
                    \node[vertex] (a15) at (-4,-0.5) {};
                    \node at (-4,-1.1) {$H(u_{p})$};
                    % $H(u_{q})$层
                    \draw[thick, fill=lightgray] (0,0) circle (0.8);
                    \node[vertex] (a21) at (0,0.5) {};
                    \node[vertex, gray] (a22) at (-0.5,0) {};
                    \node at (0,0) {$\cdots$};
                    \node[vertex, gray] (a23) at (0.5,0) {};
                    \node[solidvertex] (a25) at (0,-0.5) {};
                    \node at (0,-1.1) {$H(u_{q})$};
                    % $H(u_{w})$层
                    \draw[thick, fill=lightgray] (4,0) circle (0.8);
                    \node[vertex] (a31) at (4,0.5) {};
                    \node[vertex, gray] (a32) at (3.5,0) {};
                    \node at (4,0) {$\cdots$};
                    \node[vertex, gray] (a33) at (4.5,0) {};
                    \node[solidvertex] (a35) at (4,-0.5) {};
                    \node at (4,-1.1) {$H(u_{w})$};
                    % $H(u_{i_{1}})$层
                    \draw[thick, fill=gray] (0,2.5) circle (0.8);
                    \node[vertex] (a41) at (0,3) {};
                    \node[vertex] (a42) at (-0.5,2.5) {};
                    \node at (0,2.5) {$\cdots$};
                    \node[vertex] (a43) at (0.5,2.5) {};
                    \node[vertex] (a45) at (0,2) {};
                    \node at (1.5,2.5) {$H(u_{i_{1}})$};
                    % $H(u_{i_{2}})$层
                    \draw[thick, fill=gray] (0,-2.5) circle (0.8);
                    \node[vertex] (a51) at (0,-2) {};
                    \node[vertex] (a52) at (-0.5,-2.5) {};
                    \node at (0,-2.5) {$\cdots$};
                    \node[vertex] (a53) at (0.5,-2.5) {};
                    \node[vertex] (a55) at (0,-3) {};
                    \node at (0,-3.65) {$\vdots$};
                    \node at (1.5,-2.5) {$H(u_{i_{2}})$};
                    \node at (1.5,-3) {$\vdots$};
                    \node at (1.5,-3.5) {$\vdots$};
                    \node at (1.5,-4) {$\vdots$};
                    \node at (1.5,-4.5) {$\vdots$};
                    % $H(u_{i_{\ell}})$层
                    \draw[thick, fill=gray] (0,-5) circle (0.8);
                    \node[vertex] (a61) at (0,-4.5) {};
                    \node[vertex] (a62) at (-0.5,-5) {};
                    \node at (0,-5) {$\cdots$};
                    \node[vertex] (a63) at (0.5,-5) {};
                    \node[vertex] (a65) at (0,-5.5) {};
                    \node at (1.5,-5) {$H(u_{i_{\ell}})$};
                    % 绘制悬挂树$T_{1}$
                    \draw[base arc, dashed, blue] (a11) to (a12);
                    \draw[base arc, dashed, bend left=45, blue] (a12) to (a42);
                    \draw[base arc, dashed, bend right=30, blue] (a42) to (a22);
                    \draw[base arc, dashed, bend right=30, blue] (a42) to (a32);
                    \draw[base arc, dotted, blue] (a22) to (a25);
                    %\node at (-0.4, 0.45) [scale=0.6, blue] {$\overrightarrow{R_{1}^{(u_{q})}}$};
                    \draw[base arc, dotted, blue] (a32) to (a35);
                    % 绘制悬挂树$T_{h}$
                    \draw[base arc, dashed, orange] (a11) to (a13);
                    \draw[base arc, dashed, bend left=60, orange] (a13) to (a43);
                    \draw[base arc, dashed, bend left=45, orange] (a43) to (a23);
                    \draw[base arc, dashed, bend left=60, orange] (a43) to (a33);
                    \draw[base arc, dotted, orange] (a23) to (a25);
                    %\node at (0.12, 0) [scale=0.6, orange] {$\overrightarrow{R_{h}^{(u_{q})}}$};
                    \draw[base arc, dotted, orange] (a33) to (a35);
                    % 绘制悬挂树$T_{1}^{*}$
                    \draw[zigzag, bend left=10, green] (a11) to (a51);
                    \draw[zigzag, green] (a51) to (a52);
                    \draw[base arc, dotted, green] (a52) to (a55);
                    \draw[zigzag, bend right=30, green] (a55) to (a25);
                    \draw[zigzag, bend right=20, green] (a55) to (a35);
                    % 绘制悬挂树$T_{\ell-1}^{*}$
                    \draw[zigzag, red] (a11) to[out=0,in=90] (-2,-2.5) to[out=270,in=180] (a61);
                    \draw[zigzag, red] (a61) to (a62);
                    \draw[base arc, dotted, red] (a62) to (a65);
                    \draw[zigzag, red] (a65) to[out=210,in=180] (a25);
                    \draw[zigzag, red] (a65) to[out=-15,in=270] (a35);
                    % 绘制悬挂树$T_{\ell}^{*}$
                    \draw[base arc, bend left=60, purple] (a11) to (a41);
                    \draw[base arc, dotted, purple] (a41) to (a45);
                    \draw[base arc, bend left=90, purple] (a45) to (a25);
                    \draw[base arc, bend left=45, purple] (a45) to (a35);
                    % 标记顶点
                    \labelNode{above}{a11}{r=(u_{p},v_{a})}
                    \labelNode{below, scale=0.8}{a12}{(u_{p},v_{j_{1}})}
                    \labelNode{below, scale=0.8, xshift=2pt}{a13}{(u_{p},v_{j_{h}})}
                    \labelNode{below, scale=0.8}{a15}{(u_{p},v_{c})}
                    \labelNode{above, scale=0.8}{a21}{(u_{q},v_{a})}
                    \labelNode{below, scale=0.8}{a22}{(u_{q},v_{j_{1}})}
                    \labelNode{below, scale=0.8, xshift=2pt}{a23}{(u_{q},v_{j_{h}})}
                    \labelNode{below}{a25}{x=(u_{q},v_{c})}
                    \labelNode{above, scale=0.8}{a31}{(u_{w},v_{a})}
                    \labelNode{below, scale=0.8}{a32}{(u_{w},v_{j_{1}})}
                    \labelNode{below, scale=0.8, xshift=2pt}{a33}{(u_{w},v_{j_{h}})}
                    \labelNode{below}{a35}{y=(u_{w},v_{c})}
                    \labelNode{above, scale=0.8}{a41}{(u_{i_{1}},v_{a})}
                    \labelNode{below, scale=0.8, xshift=-2pt}{a42}{(u_{i_{1}},v_{j_{1}})}
                    \labelNode{below, scale=0.8, xshift=2pt}{a43}{(u_{i_{1}},v_{j_{h}})}
                    \labelNode{below,, scale=0.8}{a45}{(u_{i_{1}},v_{c})}
                    \labelNode{above, scale=0.8}{a51}{(u_{i_{2}},v_{a})}
                    \labelNode{below, scale=0.8, xshift=-2pt}{a52}{(u_{i_{2}},v_{j_{1}})}
                    \labelNode{below, scale=0.8, xshift=2pt}{a53}{(u_{i_{2}},v_{j_{h}})}
                    \labelNode{below,, scale=0.8}{a55}{(u_{i_{2}},v_{c})}
                    \labelNode{above, scale=0.8}{a61}{(u_{i_{\ell}},v_{a})}
                    \labelNode{below, scale=0.8, xshift=-2pt}{a62}{(u_{i_{\ell}},v_{j_{1}})}
                    \labelNode{below, scale=0.8, xshift=2pt}{a63}{(u_{i_{\ell}},v_{j_{h}})}
                    \labelNode{below,, scale=0.8}{a65}{(u_{i_{\ell}},v_{c})}
                \end{tikzpicture}
                \caption{$\ell+h$ pendant $(S,r)$-trees for Subcase 2.2 of Lemma \ref{lem3.3}.}
                \label{fig4}
            \end{figure}
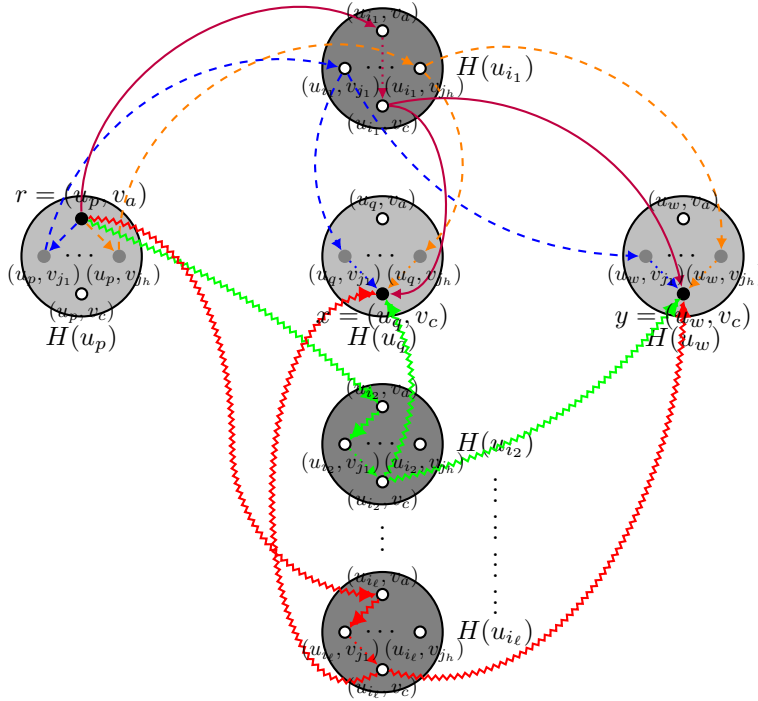
            
            For each $s\in[h]$, let $T_{s}'$ be a pendant $(S,r)$-tree induced by the arcs in%we define a tree $T_{s}'$ such that
            \begin{align*}
                A(\overrightarrow{Q_{v_{a}v_{c},s}^{(u_{p})}[v_{a},v_{j_{s}}]})\cup A(\widetilde{T}_{1}^{(v_{j_{s}})})\cup \{(u_{q},v_{j_{s}})x,(u_{w},v_{j_{s}})y\}.
            \end{align*}
            %\begin{itemize}
                %\item $V(T_{s}')=V(\overrightarrow{Q_{v_{a}v_{c},s}^{(u_{p})}[v_{a},v_{j_{s}}]})\cup V(\widetilde{T}_{1}^{(v_{j_{s}})})\cup \{x,y\}$;
                %\item $A(T_{s}')=A(\overrightarrow{Q_{v_{a}v_{c},s}^{(u_{p})}[v_{a},v_{j_{s}}]})\cup A(\widetilde{T}_{1}^{(v_{j_{s}})})\cup \{(u_{q},v_{j_{s}})x,(u_{w},v_{j_{s}})y\}$.
            %\end{itemize}
            For each $t\in [\ell-1]$, let $T_{t}^{*}$ be a pendant $(S,r)$-tree induced by the arcs in%we define a tree $T_{t}'^{*}$ such that
            \begin{align*}
                A(\widetilde{T}_{t+1}^{(v_{a})}[u_{p},u_{i_{t+1}}])\cup A(\overrightarrow{Q_{v_{a}v_{c},1}^{(u_{i_{t+1}})}})\cup A(\widetilde{T}_{t+1}^{(v_{b})}[u_{i_{t+1}},u_{q}])\cup A(\widetilde{T}_{t+1}^{(v_{c})}[u_{i_{t+1}},u_{w}]).
            \end{align*}
            %\begin{itemize}
                %\item $V(T_{t}'^{*})=V(\overrightarrow{P_{u_{p}u_{q},t+1}^{(v_{a})}[u_{p},u_{i_{t+1}}]})\cup V(\overrightarrow{Q_{v_{a}v_{c},1}^{(u_{i_{t+1}})}})\cup V(\overrightarrow{P_{u_{p}u_{q},t+1}^{(v_{c})}[u_{i_{t+1}},u_{q}]})\cup V(\overrightarrow{P_{u_{p}u_{w},t+1}^{(v_{c})}[u_{i_{t+1}},u_{w}]})$;
                %\item $A(T_{t}'^{*})=A(\overrightarrow{P_{u_{p}u_{q},t+1}^{(v_{a})}[u_{p},u_{i_{t+1}}]})\cup A(\overrightarrow{Q_{v_{a}v_{c},1}^{(u_{i_{t+1}})}})\cup A(\overrightarrow{P_{u_{p}u_{q},t+1}^{(v_{c})}[u_{i_{t+1}},u_{q}]})\cup A(\overrightarrow{P_{u_{p}u_{w},t+1}^{(v_{c})}[u_{i_{t+1}},u_{w}]})$.
            %\end{itemize}
            Finally, let $T_{\ell}^{*}$ be a pendant $(S,r)$-tree induced by the arcs in
            \begin{align*}
                A(\widetilde{T}_{1}^{(v_{a})}[u_{p},u_{i_{1}}])\cup A(\overrightarrow{Q_{v_{a}v_{c},h+1}^{(u_{i_{1}})}})\cup A(\widetilde{T}_{1}^{(v_{b})}[u_{i_{1}},u_{q}]) \cup A(\widetilde{T}_{1}^{(v_{c})}[u_{i_{1}},u_{w}]).
            \end{align*}
            
            For clarity, the representative trees $T_{1}',T_{h}',T_{1}^{*},T_{\ell-1}^{*}$ and $T_{\ell}^{*}$ are depicted in Figures~\ref{fig3} and~\ref{fig4}, where blue dashed lines represent $T_{1}'$, orange dashed lines represent $T_{h}'$, green zigzag lines represent $T_{1}^{*}$, red zigzag lines represent $T_{\ell-1}^{*}$, and purple lines represent $T_{\ell}^{*}$. Note that dotted lines denote arcs, while all other lines represent directed paths. It can be verified that across both subcases, the trees $T_{1}^{*},T_{2}^{*},\cdots,T_{\ell}^{*}$ and $T_{1}',T_{2}',\cdots,T_{h}'$ together form $\ell+h$ pairwise internally-disjoint pendant $(S,r)$-trees in $D\square H$.
        \end{SubProofCase}
        %Hence, we have
            %\begin{align*}
                %\tau_{S,r}(D\square H)\geq h+\ell\geq 3\lfloor\frac{\ell}{2}\rfloor,
            %\end{align*}
            %where we use the earlier assumption that $h\geq \ell\geq 1$ for the last inequality.
    \end{ProofCase}
    
    \begin{ProofCase}
        $v_{a},v_{b}$ and $v_{c}$ are the same vertex in $H$. By Lemma \ref{lem2.2}, $r$ has at least $\ell+2$ out-neighbors in $H(u_{p})$, denoted by $(u_{p},v_{j_{1}}),(u_{p},v_{j_{2}}),\cdots,(u_{p},v_{j_{h+2}})$. In the sequel, we only use $\ell+1$ of these out-neighbors. By the same reason as in Subcase 2.1, we obtain $\overrightarrow{R_{s}^{(u_{q})}}$ and $\overrightarrow{R_{s}^{(u_{w})}}$ for $s\in[h+1]$, where $\overrightarrow{R_{s}^{(u_{q})}}$ is the directed $(u_{q},v_{j_{s}})-x$ path in $H(u_{q})$ and $\overrightarrow{R_{s}^{(u_{w})}}$ is the directed $(u_{q},v_{j_{s}})-y$ path in $H(u_{w})$. 
        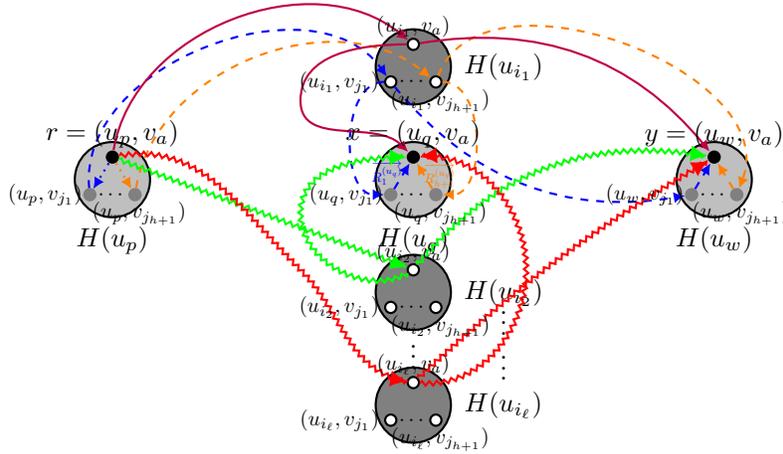
\begin{figure}[H] 
            \centering
            \begin{tikzpicture}
                % $H(u_{p})$层
                \draw[thick, fill=lightgray] (-4,0) circle (0.5);
                \node[solidvertex] (a11) at (-4,0.3) {};
                \node[vertex, gray] (a12) at (-4.3,-0.2) {};
                \node[scale=0.8] at (-4,-0.2) {$\cdots$};
                \node[vertex, gray] (a13) at (-3.7,-0.2) {};
                \node at (-4,-0.8) {$H(u_{p})$};
                % $H(u_{q})$层
                \draw[thick, fill=lightgray] (0,0) circle (0.5);
                \node[solidvertex] (a21) at (0,0.3) {};
                \node[vertex, gray] (a22) at (-0.3,-0.2) {};
                \node[scale=0.8] at (0,-0.2) {$\cdots$};
                \node[vertex, gray] (a23) at (0.3,-0.2) {};
                \node at (0,-0.8) {$H(u_{q})$};
                % $H(u_{w})$层
                \draw[thick, fill=lightgray] (4,0) circle (0.5);
                \node[solidvertex] (a31) at (4,0.3) {};
                \node[vertex, gray] (a32) at (3.7,-0.2) {};
                \node[scale=0.8] at (4,-0.2) {$\cdots$};
                \node[vertex, gray] (a33) at (4.3,-0.2) {};
                \node at (4,-0.8) {$H(u_{w})$};
                % $H(u_{i_{1}})$层
                \draw[thick, fill=gray] (0,1.5) circle (0.5);
                \node[vertex] (a41) at (0,1.8) {};
                \node[vertex] (a42) at (-0.3,1.3) {};
                \node[scale=0.8] at (0,1.3) {$\cdots$};
                \node[vertex] (a43) at (0.3,1.3) {};
                \node at (1.2,1.5) {$H(u_{i_{1}})$};
                % $H(u_{i_{2}})$层
                \draw[thick, fill=gray] (0,-1.5) circle (0.5);
                \node[vertex] (a51) at (0,-1.2) {};
                \node[vertex] (a52) at (-0.3,-1.7) {};
                \node[scale=0.8] at (0,-1.7) {$\cdots$};
                \node[vertex] (a53) at (0.3,-1.7) {};
                \node at (0,-2.25) {$\vdots$};
                \node at (1.2,-1.5) {$H(u_{i_{2}})$};
                \node at (1.2,-1.8) {$\vdots$};
                \node at (1.2,-2.4) {$\vdots$};
                % $H(u_{i_{\ell}})$层
                \draw[thick, fill=gray] (0,-3) circle (0.5);
                \node[vertex] (a61) at (0,-2.7) {};
                \node[vertex] (a62) at (-0.3,-3.2) {};
                \node[scale=0.8] at (0,-3.2) {$\cdots$};
                \node[vertex] (a63) at (0.3,-3.2) {};
                \node at (1.2,-3) {$H(u_{i_{\ell}})$};
                % 绘制悬挂树$T_{1}$
                \draw[base arc, dotted, blue] (a11) to (a12);
                \draw[base arc, dashed, bend left=75, blue] (a12) to (a42);
                \draw[base arc, dashed, bend right=90, blue] (a42) to (a22);
                \draw[base arc, dashed, bend right=30, blue] (a42) to (a32);
                \draw[base arc, dashed, blue] (a22) to (a21);
                \node at (-0.35, 0.08) [scale=0.5, blue] {$\overrightarrow{R_{1}^{(u_{q})}}$};
                \draw[base arc, dashed, blue] (a32) to (a31);
                % 绘制悬挂树$T_{h}$
                \draw[base arc, dotted, orange] (a11) to (a13);
                \draw[base arc, dashed, bend left=60, orange] (a13) to (a43);
                \draw[base arc, dashed, bend left=90, orange] (a43) to (a23);
                \draw[base arc, dashed, bend left=90, orange] (a43) to (a33);
                \draw[base arc, dashed, orange] (a23) to (a21);
                \node at (0.35, 0.05) [scale=0.5, orange] {$\bf{\overrightarrow{R_{h+1}^{(u_{q})}}}$};
                \draw[base arc, dashed, orange] (a33) to (a31);
                % 绘制悬挂树$T_{\ell}^{*}$
                \draw[base arc, bend left=60, purple] (a11) to (a41);
                \draw[base arc, purple] (a41) to[out=180,in=90] (-1.5,1.05) to[out=270,in=135] (a21);
                \draw[base arc, bend left=30, purple] (a41) to (a31);
                % 绘制悬挂树$T_{2}^{*}$
                \draw[zigzag, green] (a11) to (a51);
                \draw[zigzag, green] (a51) to[out=210,in=270] (-1.5,-0.45) to[out=90,in=180] (a21);
                \draw[zigzag, bend left=30, green] (a51) to (a31);
                % 绘制悬挂树$T_{\ell}^{*}$
                \draw[zigzag, red] (a11) to[out=10,in=135] (-1.5,-1.2) to[out=315,in=170] (a61);
                \draw[zigzag, red] (a61) to[out=-5,in=270] (1.5,-1.2) to[out=90,in=5] (a21);
                \draw[zigzag, bend left=6, red] (a61) to (a31);
                % 标记顶点
                \labelNode{above}{a11}{r=(u_{p},v_{a})}
                \labelNode{left, scale=0.8}{a12}{(u_{p},v_{j_{1}})}
                \labelNode{below, scale=0.8, xshift=2pt}{a13}{(u_{p},v_{j_{h+1}})}
                \labelNode{above}{a21}{x=(u_{q},v_{a})}
                \labelNode{left, scale=0.8}{a22}{(u_{q},v_{j_{1}})}
                \labelNode{below, scale=0.8, xshift=2pt}{a23}{(u_{q},v_{j_{h+1}})}
                \labelNode{above}{a31}{y=(u_{w},v_{a})}
                \labelNode{left, scale=0.8}{a32}{(u_{w},v_{j_{1}})}
                \labelNode{below, scale=0.8, xshift=2pt}{a33}{(u_{w},v_{j_{h+1}})}
                \labelNode{above, scale=0.8}{a41}{(u_{i_{1}},v_{a})}
                \labelNode{left, scale=0.8, xshift=-2pt}{a42}{(u_{i_{1}},v_{j_{1}})}
                \labelNode{below, scale=0.8, xshift=2pt}{a43}{(u_{i_{1}},v_{j_{h+1}})}
                \labelNode{above, scale=0.8}{a51}{(u_{i_{2}},v_{a})}
                \labelNode{left, scale=0.8, xshift=-2pt}{a52}{(u_{i_{2}},v_{j_{1}})}
                \labelNode{below, scale=0.8, xshift=2pt}{a53}{(u_{i_{2}},v_{j_{h+1}})}
                \labelNode{above, scale=0.8}{a61}{(u_{i_{\ell}},v_{a})}
                \labelNode{left, scale=0.8, xshift=-2pt}{a62}{(u_{i_{\ell}},v_{j_{1}})}
                \labelNode{below, scale=0.8, xshift=2pt}{a63}{(u_{i_{\ell}},v_{j_{h+1}})}
            \end{tikzpicture}
            \caption{$\ell+h+1$ pendant $(S,r)$-trees for Case 3 of Lemma \ref{lem3.3}.}
            \label{fig5}
        \end{figure}
        For each $s\in[h]$, let $T_{s}'$ be a pendant $(S,r)$-tree induced by the arcs in
        \begin{align*}
            \{r(u_{p},v_{j_{s}})\}\cup V(\widetilde{T}_{1}^{(v_{j_{s}})})\cup A(\overrightarrow{R_{s}^{(u_{q})}})\cup A(\overrightarrow{R_{s}^{(u_{w})}}).
        \end{align*}
        As a representative example, $T_{1}'$ and $T_{h}'$ are depicted by the blue and orange dashed lines in Figure \ref{fig5}.
        %\begin{itemize}
            %\item $V(T_{s})=\{r\}\cup V(\widetilde{T}_{1}^{(v_{j_{s}})})\cup V(\overrightarrow{R_{s}^{(u_{q})}})\cup V(\overrightarrow{R_{s}^{(u_{w})}})$;
            %\item $A(T_{s})=\{r(u_{p},v_{j_{s}})\}\cup V(\widetilde{T}_{1}^{(v_{j_{s}})})\cup A(\overrightarrow{R_{s}^{(u_{q})}})\cup A(\overrightarrow{R_{s}^{(u_{w})}})$.
        %\end{itemize}
        
        Note that the trees $\widetilde{T}_{1}^{(v_{a})},\widetilde{T}_{2}^{(v_{a})},\cdots,\widetilde{T}_{\ell}^{(v_{a})}$ are also pendant $(S,r)$-trees. For visual clarity, only the first, second, and last of these trees are shown in Figure \ref{fig5}, represented by the purple lines, green zigzag lines, and red zigzag lines, respectively.
        Together with $T_{1}',T_{2}',\cdots,T_{h+1}'$, these trees form $\ell+h+1$ pairwise internally-disjoint pendant $(S,r)$-trees in $D\square H$. For reference, dotted lines denote arcs, while all other lines represent directed paths.%Figure \ref{fig5} illustrates the pendant trees $T_{1},T_{h},T_{1}^{(v_{a})},T_{\ell-1}^{(v_{a})}$ and $T_{\ell}^{(v_{a})}$, where blue dashed lines represent $T_{1}$, orange dashed lines represent $T_{h}$, purple lines represent $\widetilde{T}_{1}^{(v_{a})}$, green zigzag lines represent $\widetilde{T}_{2}^{(v_{a})}$, and red zigzag lines represent $\widetilde{T}_{\ell}^{(v_{a})}$. 
        %Note that dotted lines represent arcs, and all other lines represent directed paths.
        %Hence, we have
        %\begin{align*}
            %\tau_{S,r}(D\square H)\geq h+\ell+1\geq 3\lfloor\frac{\ell}{2}\rfloor,
        %\end{align*}
        %where we use the earlier assumption that $h\geq \ell\geq 1$ for the last inequality.
    \end{ProofCase}
    Based on the above arguments, this lemma is thus proved.
\end{proof}

\specialproof{thm1.1}
    By Lemmas \ref{lem3.1}, \ref{lem3.2} and \ref{lem3.3}, for any terminal vertex set $S=\{r,x,y\}\subseteq V(D\square H)$, we have
    \begin{align*}
        \tau_{3}(D\square H)=\min_{S\subseteq V(D\square H),r\in S} \tau_{S,r}(D\square H)\geq \ell+h=\tau_{3}(D)+\tau_{3}(H).
        %\tau_{3}(D\square H)=\min_{S\subseteq V(D\square H),r\in S} \tau_{S,r}(D\square H)\geq 3\lfloor\frac{\ell}{2}\rfloor=3\lfloor\frac{\tau_{3}(D)}{2}\rfloor
    \end{align*}
    %By the symmetry of Cartesian products, this also holds for $3\lfloor\frac{\tau_{3}(H)}{2}\rfloor$. Thus, $\tau_{3}(D\square H)\geq \min \{3\lfloor\frac{\tau_{3}(D)}{2}\rfloor,3\lfloor\frac{\tau_{3}(H)}{2}\rfloor\}.$
    To demonstrate the sharpness of this lower bound, we consider the following example.
    \begin{example}
        Let $n$ and $m$ be two integers with $n\geq 3$ and $m\geq 3$. The Cartesian product digraph $\overleftrightarrow{P_{n}}\square \overrightarrow{C_{m}}$ has no pendant $(S,r)$-trees, where $r\in S\subseteq V(\overleftrightarrow{P_{n}}\square \overrightarrow{C_{m}})$ and $|S|=3$.
        
        %\begin{proof}
            Obviously, $\overleftrightarrow{P_{n}}$ and $\overrightarrow{C_{m}}$ are both strong. Without loss of generality, let $V(\overleftrightarrow{P_{n}})$ and $V(\overrightarrow{C_{m}})$ have the following vertex and arc sets
            \begin{itemize}
                \item $V(\overleftrightarrow{P_{n}})=\{u_{i}\mid i\in [n]\}$, $A(\overleftrightarrow{P_{n}})=\{u_{i_{1}}u_{i_{2}}\mid |i_{1}-i_{2}|=1\}$;
                \item $V(\overrightarrow{C_{m}})=\{v_{j}\mid j\in [m]\}$, $A(\overrightarrow{C_{m}})=\{v_{j_{1}}v_{j_{2}}\mid j_{2}-j_{1}=1~(\text{mod}~m)\}$.
            \end{itemize}
            For short, let $Q=\overleftrightarrow{P_{n}}\square \overrightarrow{C_{m}}$, %On one hand, observe that $\tau_{3}(\overleftrightarrow{P_{n}})=0$ and $\tau_{3}(\overrightarrow{C_{m}})=0$. Hence, by Theorem \ref{thm1.1}, $\tau_{3}(Q)}\geq 0$. On the other hand, 
            $S=\{(u_{1},v_{1}),(u_{1},v_{2}),(u_{2},v_{1})\}\subseteq V(Q)$ and $r=(u_{1},v_{1})$ be the root. Clearly, $d_{Q}^{+}(r)=2$. By the definition of a pendant $(S,r)$-tree, neither of the arcs $r(u_{1},v_{2}), r(u_{2},v_{1})$ can belong to a pendant $(S, r)$-tree, so $\tau_{3}(Q)\leq 0$. Hence, we conclude that $\tau_{3}(Q)=0$.
        %\end{proof}
    \end{example}
    This completes the proof of Theorem \ref{thm1.1}.
\endspecialproof

\section{A polynomial-time algorithm for finding internally-disjoint pendant  \texorpdfstring {$(S,r)$-trees}{(S,r)-trees}}
In this section, we will propose a polynomial-time algorithm to find a set of internally-disjoint pendant $(S,r)$-trees in Cartesian product digraphs $D\square H$, which always attain the lower bound of Theorem \ref{thm1.1}. For short, we denote by $x_{p,a}$ the vertex $(u_{p},v_{a})$ in $D\square H$, and refer to internally-disjoint directed paths as ``IDDPs".

\begin{breakablealgorithm}\label{Alg-1}
\caption{\sc Internally-disjoint Pendant $(S,x_{p,a})$-Trees Searching Algorithm} %算法标题
\begin{algorithmic}[1] % 一行一个标行号
    \Require
       Two strong digraphs $D$, $H$ with $\tau_{3}(D)=\ell\geq 1$ and $\tau_{3}(H)=h\geq 1$, a vertex set $S=\{x_{p,a},x_{q,b},x_{w,c}\}$, $\ell$ internally-disjoint pendant $(S_{D},u_{p})$-trees $\widetilde{T}_{1},\cdots,\widetilde{T}_{\ell}$, $h$ internally-disjoint pendant $(S_{H},v_{a})$-trees $\hat{T}_{1},\cdots,\hat{T}_{h}$.
    \Ensure
       $\ell+h$ internally-disjoint pendant $(S,x_{p,a})$-trees $T_{1}',\cdots,T_{h}'$, $T_{1}^{*},\cdots,T_{\ell}^{*}$.
    \State {$S_{D}\leftarrow \{u_{p},u_{q},u_{w}\}$, $S_{H}\leftarrow \{v_{a},v_{b},v_{c}\}$.}
    \If {$|S_{D}|=1$}
        \State {$Z\leftarrow \{z_{1},\cdots,z_{\ell}\}$, where $z_{i}\in\{u_{j}\mid u_{p}u_{j}\in A(D)\}$ and they are distinct.}
        \State {$\{R_{i}: z_{i}-u_{p}\}_{i=1}^{\ell}\leftarrow$ $\ell$-fan in $D$ by \textbf{Alg 2}.}
        \For {$1\leq i\leq \ell$}
            \State{Choose $R_{i}^{(v_{b})}$, $R_{i}^{(v_{c})}$, $\hat{T}_{1}^{(z_{i})}$ and $x_{p,a}(z_{i},v_{a})$ to construct $T_{i}^{*}$.}
        \EndFor
        \State {$T_{j}'\leftarrow \hat{T}_{j}^{(u_{p})}$ for $1\leq j\leq h$.}
    \EndIf
    \If {$|S_{D}|=2$}
        \State {Let $(f,g),(f',g')$ be an ordering of $\{(q,b),(w,c)\}$ such that $f'\neq p$.}
        \State {$\{P_{i}: u_{p}-u_{f'}\}_{i=1}^{\ell+1}\leftarrow$ IDDPs in $D$ by \textbf{Alg 3}, with $|P_{\ell+1}|$ shortest.}
        \State {$z_{i}\leftarrow$ the predecessor of $u_{f'}$ on the path $P_{i}$ for each $1\leq i\leq \ell$.}
        \State {$\{R_{i}: z_{i}-u_{p}\}_{i=1}^{\ell}\cup \{R_{\ell+1}: u_{f'}-u_{p}\}\leftarrow$ $(\ell+1)$-fan in $D$ by \textbf{Alg 2}.}
        \If {$f=p$}
            \If {$g'\notin \{a,g\}$}
                \For {$1\leq j\leq h$}
                    \State {$s_{j}\leftarrow$ the predecessor of $v_{g'}$ on the tree $\hat{T}_{j}$.}
                    \State{Choose $\hat{T}_{j}^{(u_{p})}$, $P_{\ell+1}^{(s_{j})}$ and $(u_{f'},s_{j})x_{f',g'}$ to construct $T_{j}'$.}
                \EndFor
                \For {$1\leq i\leq \ell$}
                    \State {Choose $P_{i}^{(v_{a})}$, $\hat{T}_{1}^{(z_{i})}$, $R_{i}^{(v_{g})}$ and $(z_{i},v_{g'})x_{f',g'}$ to construct $T_{i}^{*}$.}
                \EndFor
            \Else
                \State {$\{Q_{j}: v_{a}-v_{g}\}_{j=1}^{h+1}\leftarrow$ IDDPs in $H$ by \textbf{Alg 3}, with $|Q_{h+1}|$ shortest.}
                \State{$s_{j}\leftarrow$ the predecessor of $v_{g}$ on the path $Q_{j}$ for $1\leq j\leq h$.}
                \If {$g'=g$}
                    \For {$1\leq j\leq h$}
                        \State{Choose $Q_{j}^{(u_{p})}$, $P_{\ell+1}^{(s_{j})}$ and $(u_{f'},s_{j})x_{f',g'}$ to construct $T_{j}'$.}
                    \EndFor
                    \For {$1\leq i\leq \ell$}
                        \State {Choose $P_{i}^{(v_{a})}$, $Q_{1}^{(z_{i})}$, $R_{i}^{(v_{g})}$, $(z_{i},v_{g'})x_{f',g'}$ to construct $T_{i}^{*}$.}
                    \EndFor
                \Else
                    \State {$\{W_{j}: s_{j}-v_{a}\}_{j=1}^{h}\cup \{W_{h+1}: v_{g}-v_{a}\}\leftarrow$ $(h+1)$-fan by \textbf{Alg 2}.}
                    \For {$1\leq j\leq h$}
                        \State{Choose $Q_{j}^{(u_{p})}$, $P_{\ell+1}^{(s_{j})}$ and $W_{j}^{(u_{f'})}$ to construct $T_{j}'$.}
                    \EndFor
                    \For {$1\leq i\leq \ell$}
                        \State {Choose $P_{i}^{(v_{a})}$, $Q_{1}^{(z_{i})}$ and $R_{i}^{(v_{g})}$ to construct $T_{i}^{*}$.}
                    \EndFor
                \EndIf
            \EndIf
        \Else
            \If {$a\notin \{g,g'\}$}
                \For {$1\leq j\leq h$}
                    \State {$s_{j}\leftarrow$ the branch vertex of $\hat{T}_{j}$.}
                    \State {Choose $\hat{T}_{j}^{(u_{p})}$, $P_{\ell+1}^{(s_{j})}$ and $\hat{T}_{j}^{(u_{w})}$ to construct $T_{j}'$.}
                \EndFor
                \For {$1\leq i\leq \ell$}
                    \State {Choose $P_{i}^{(v_{a})}$, $\hat{T}_{1}^{(z_{i})}$, $(z_{i},v_{g})x_{f,g}$, $(z_{i},v_{g'})x_{f',g'}$ to construct $T_{j}'$.}
                \EndFor
            \Else
                \State {Let $(d,e),(d',e')$ be an ordering of $\{(f,g),(f',g')\}$ such that $e=a$}
                \State {$\{Q_{j}:v_{a}-v_{e'}\}_{j=1}^{h+1}\leftarrow$ IDDPs in $H$ by \textbf{Alg 3}, with $|Q_{h+1}|$ shortest.}
                \State {$s_{j}\leftarrow$ the predecessor of $v_{e'}$ on the path $Q_{j}$ for $1\leq j\leq h$.}
                \State {$\{W_{j}:s_{j}-v_{e}\}_{j=1}^{h}\cup \{W_{h+1}:v_{e'}-v_{e}\}\leftarrow(h+1)$-fan by \textbf{Alg 2}.}
                \For {$1\leq j\leq h$}
                    \State {Choose $Q_{j}^{(u_{p})}$, $P_{\ell+1}^{(s_{j})}$, $W_{j}^{(u_{d'})}$, $(u_{d'},s_{j})x_{d',e'}$ to construct $T_{j}'$.}
                \EndFor
                \For {$1\leq i\leq \ell$}
                    \State {Choose $P_{i}^{(v_{a})}$, $Q_{1}^{(z_{i})}$, $(z_{i},v_{e'})x_{d',e'}$ to construct $T_{i}^{*}$.}
                \EndFor
            \EndIf
        \EndIf
    \EndIf
    \If {$|S_{D}|=3$}
        \State {$z_{i}\leftarrow$ the branch vertex of $\widetilde{T}_{i}$ for $1\leq i\leq \ell$.}
        \If {$|S_{H}|=3$}
            \State {$s_{j}\leftarrow$ the branch vertex of $\hat{T}_{j}$ for $1\leq j\leq h$.}
            \State {$\{W_{j}: s_{j}-v_{b}\}_{j=1}^{h}\cup \{W_{h+1}: v_{a}-v_{b}\}\leftarrow$ $(h+1)$-fan in $H$ by \textbf{Alg 2}.}
            \State {$\{Y_{j}: s_{j}-v_{c}\}_{j=1}^{h}\cup \{Y_{h+1}: v_{a}-v_{c}\}\leftarrow$ $(h+1)$-fan in $H$ by \textbf{Alg 2}.}
            \For {$1\leq j\leq h$}
                \State {Choose $\hat{T}_{j}^{(u_{p})}$, $\widetilde{T}_{1}^{(s_{j})}$, $W_{j}^{(u_{q})}$ and $Y_{j}^{(u_{w})}$ to construct $T_{j}'$.}
            \EndFor
            \For {$1\leq i\leq \ell-1$}
                \State {Choose $\widetilde{T}_{i+1}^{(v_{a})}$, $\hat{T}_{1}^{(z_{i+1})}$, $\widetilde{T}_{i+1}^{(v_{b})}$ and $\widetilde{T}_{i+1}^{(v_{c})}$ to construct $T_{i}^{*}$.}
            \EndFor
            \State {Choose $\widetilde{T}_{1}^{(v_{a})}$, $W_{h+1}^{(u_{q})}$ and $Y_{h+1}^{(u_{w})}$ to construct $T_{\ell}^{*}$.}
        \ElsIf {$|S_{H}|=2$}
            \State {Let $(f,g),(f',g')$ be an ordering of $\{(q,b),(w,c)\}$ such that $g'\neq a$.}
            \State {$\{Q_{j}: v_{a}-v_{g'}\}_{j=1}^{h+1}\leftarrow$ IDDPs in $H$ by \textbf{Alg 3}, with $|Q_{h+1}|$ shortest.}
            \State{$s_{j}\leftarrow$ the predecessor of $v_{g'}$ on the path $Q_{j}$ for $1\leq j\leq h$.}
            \If {$a=g$}
                \State {$\{W_{j}: s_{j}-v_{a}\}_{j=1}^{h}\cup \{W_{h+1}: v_{g'}-v_{a}\}\leftarrow$ $(h+1)$-fan by \textbf{Alg 2}.}
                \For {$1\leq j\leq h$}
                    \State {Choose $Q_{j}^{(u_{p})}$, $\widetilde{T}_{1}^{(s_{j})}$, $W_{j}^{(u_{f})}$, $(u_{f'},s_{j})x_{f',g'}$ to construct $T_{j}'$.}
                \EndFor
                \For {$1\leq i\leq \ell-1$}
                    \State {Choose $\widetilde{T}_{i+1}^{(v_{a})}$, $Q_{\ell+1}^{(z_{i+1})}$, $\widetilde{T}_{i+1}^{(v_{g'})}$ to construct $T_{i}^{*}$.}
                \EndFor
                \State {Choose $\widetilde{T}_{1}^{(v_{a})}$, $Q_{\ell+1}^{(u_{f'})}$ to construct $T_{\ell}^{*}$.}
            \Else
                \For {$1\leq j\leq h$}
                    \State {Choose $Q_{j}^{(u_{p})}$, $\widetilde{T}_{1}^{(s_{j})}$, $(u_{q},s_{j})x_{q,b},(u_{w},s_{j})x_{w,c}$ to construct $T_{j}'$.}
                \EndFor
                \For {$1\leq i\leq \ell-1$}
                    \State {Choose $\widetilde{T}_{i+1}^{(v_{a})}$, $Q_{1}^{(z_{i+1})}$, $\widetilde{T}_{i+1}^{(v_{b})}$ to construct $T_{i}^{*}$.}
                \EndFor
                \State {Choose $\widetilde{T}_{1}^{(v_{a})}$, $Q_{h+1}^{(v_{b})}$, $Q_{h+1}^{(v_{c})}$ to construct $T_{\ell}^{*}$.}
            \EndIf
        \Else
            \State{Choose $h$ distinct vertices $s_{1},\cdots,s_{h}\in \{v_{j}\in V(H)\mid v_{a}v_{j}\in A(H)\}$.}
            \State {$\{W_{j}: s_{j}-v_{a}\}_{j=1}^{h}\leftarrow$ $h$-fan in $H$ by \textbf{Alg 2}.}
            \For {$1\leq j\leq h$}
                \State {Choose $x_{p,a}(u_{p},s_{j})$, $\widetilde{T}_{1}^{(s_{j})}$, $W_{j}^{(u_{q})}$ and $W_{j}^{(u_{w})}$ to construct $T_{j}'$.}
            \EndFor
            \State {$T_{i}^{*}\leftarrow \widetilde{T}_{i}^{(v_{a})}$ for each $1\leq i\leq \ell$.}
        \EndIf
    \EndIf
\end{algorithmic}                               
\end{breakablealgorithm}

%We use the following algorithms to find an $\ell$-fan and $\ell$ IDDPs in an  $\ell$-strong digraph $D$.
Algorithm \ref{Alg-1} uses two polynomial-time subroutines. The first finds an $\ell$-fan, and the second finds $\ell$ internally-disjoint directed paths between two vertices. These are implemented by Algorithms \ref{Alg-2} and \ref{Alg-3}, respectively. Both subroutines are based on standard network-flow constructions.
\begin{breakablealgorithm}\label{Alg-2}
\caption{\sc An $\ell$-fan Searching Algorithm} %算法标题
\begin{algorithmic}[1] % 一行一个标行号
    \Require
       An $\ell$-strong digraph $D$, a vertex $u\in V(D)$, a set $Z=\{z_{1},z_{2},\cdots,z_{\ell}\}\subseteq V(D)\backslash \{u\}$.
    \Ensure
       An $\ell$-fan $\{R_{i}:z_{i}-u\}_{i=1}^{\ell}$ in $D$.
    \State{Construct a flow network $N$ from digraph $D$.}
    \State{Transform each vertex $x\in V(D)$ into $x_{in}$ and $x_{out}$.}
    \State{For each arc $x_{in}x_{out}$, set capacity $1$~($x\neq u$) or $\ell$~($x=u$).}
    \State{For each arc $xy\in A(D)$, set capacity $\ell$ to arc $x_{out}y_{in}$.}
    \State{Add a super-source $s$, connect $s$ to each $z_i\in Z$ with capacity $1$, and set the sink as $u_{out}$.}
    \State{Apply the \textbf{Dinic's algorithm} to compute the maximum flow from $s$ to $u_{out}$ in $N$.}
    \State{Return $\ell$-fan $\{R_{i}:z_{i}-u\}_{i=1}^\ell$ in $D$ based on the flow distribution of the maximum flow.}
\end{algorithmic}                               
\end{breakablealgorithm}

Similarly, we can find $\ell$ internally-disjoint directed paths between two distinct vertices in digraph $D$ by the following algorithm.
\begin{breakablealgorithm}\label{Alg-3}
\caption{\sc Internally-disjoint directed paths Searching Algorithm} %算法标题
\begin{algorithmic}[1] % 一行一个标行号
    \Require
       An $\ell$-strong digraph $D$, two distinct vertices $u,v\in V(D)$.
    \Ensure
       $\ell$ IDDPs $\{P_{i}:u-v\}_{i=1}^{\ell}$ in $D$.
    \State{Construct a flow network $N$ from digraph $D$.}
    \State{Transform each vertex $x\in V(D)$ into $x_{in}$ and $x_{out}$.}
    \State{For each arc $x_{in}x_{out}$, set capacity $1$~($x\notin \{u,v\}$) or $\ell$~($x=u~\text{or}~v$).}
    \State{For each arc $xy\in A(D)$, set capacity $\ell$ to arc $x_{out}y_{in}$.}
    \State{Set the source as $u_{out}$ and the sink as $v_{in}$.}
    \State{Apply the \textbf{Dinic's algorithm} to compute the maximum flow from $u_{out}$ to $v_{in}$ in $N$.}
    \State{Return $\ell$ IDDPs $\{P_{i}:u-v\}_{i=1}^{\ell}$ in $D$ based on the flow distribution of the maximum flow.}
\end{algorithmic}                               
\end{breakablealgorithm}

We are now give the proof of Theorem~\ref{thm1.2}, which shows the time-complexity of Algorithm~\ref{Alg-1}. Before this, we need the following lemma.

\begin{lemma}\cite{Dinic}\label{Dinic}
    Let $D=(V(D),A(D))$ be a digraph. \textbf{Dinic's algorithm} correctly determines a maximum flow in time $O(|V(D)|^{2}|A(D)|)$.
\end{lemma}

\specialproof{thm1.2}
    %When $|S_{D}|=1$, it takes $O(|V(D)^{2}||A(D)|)$ to  construct the pendant $(S,r)$-trees.
    %It needs $O(|V(H)^{2}||A(H)|)$ to find an $\ell+1$ IDPPs in Step 12 and Step 14, and $O(|V(D)^{2}||A(D)|)$ to find an $\ell$-fan in Step 24 and Step 34. It takes $O(|V(D)^{2}||A(D)|+|V(H)^{2}||A(H)|)$ to construct the pendant $(S,r)$-trees when $|S_{D}|=2$. When $|S_{D}|=3$, it also takes 
    No matter which of the three cases $|S_{D}|=1,2~\text{or}~3$ is executed, Algorithm~\ref{Alg-1} invokes Algorithm \ref{Alg-2} and \ref{Alg-3} at most a constant number of times on $D$ and $H$. %Each invocation of Algorithm \ref{Alg-2} or \ref{Alg-3} on $D$ (resp. $H$) is performed to find $\ell$ (resp. $h$) internally-disjoint directed paths or to construct an $\ell$-fan, which corresponds to computing a maximum flow of value $\ell$. By Lemma \ref{}
    Each such operation takes $O(|V(D)|^{2}|A(D)|)$ when performed on $D$, or $O(|V(H)|^{2}|A(H)|)$ when performed on $H$. Therefore, the overall time complexity of Algorithm \ref{Alg-1} is $O(|V(D)|^{2}|A(D)|+|V(H)|^{2}|A(H)|)$.
\endspecialproof

\vskip 1cm

\noindent {\bf Acknowledgement.} This work was supported by National Natural Science Foundation of China under Grant No. 12371352 and Yongjiang Talent Introduction Programme of Ningbo under Grant No. 2021B-011-G.

\end{document}